\documentclass[a4paper,12pt,english,francais]{amsart}

\usepackage{ifpdf}
\ifpdf\usepackage{hyperref}\else\usepackage[hypertex]{hyperref}\fi
\usepackage{hyperref}

\usepackage{amssymb,calrsfs,babel}\NoAutoSpaceBeforeFDP

\usepackage{lmodern}

\usepackage[T1]{fontenc}

\newcommand{\setC}{\mathbb{C}}
\newcommand{\setN}{\mathbb{N}}
\newcommand{\setR}{\mathbb{R}}

\newcommand{\bv}{\mathbf{v}}
\newcommand{\cC}{\mathcal{C}}
\newcommand{\cD}{\mathcal{D}}

\newcommand{\cH}{\mathcal{H}}
\newcommand{\cL}{\mathcal{L}}
\newcommand{\cM}{\mathcal{M}}
\newcommand{\cN}{\mathcal{N}}

\newcommand{\cP}{\mathcal{P}}
\newcommand{\cQ}{\mathcal{Q}}

\newcommand{\cT}{\mathcal{T}}

\renewcommand{\leq}{\leqslant}
\renewcommand{\geq}{\geqslant}

\newcommand{\db}{\bar \partial}

\DeclareMathOperator{\vol}{vol}

\DeclareMathOperator{\FS}{FS}

\DeclareMathOperator{\im}{Im}

\DeclareMathOperator{\re}{Re}

\DeclareMathOperator{\Sym}{Sym}

\newtheorem{theo}{Th\'eor\`eme}
\newtheorem{prop}[theo]{Proposition}
\newtheorem{lemm}[theo]{Lemme}

\theoremstyle{definition}
\newtheorem{defi}[theo]{D\'efinition}

\theoremstyle{remark}
\newtheorem{rema}[theo]{Remarque}
\newtheorem*{rema*}{Remarque}
\newtheorem*{ques}{Question}

\begin{document}

\author{Olivier Biquard}
\title{M\'etriques hyperk\"ahl\'eriennes pli\'ees}
\address{UPMC Universit\'e Paris 6 et \'Ecole Normale Sup\'erieure, UMR 8553 du CNRS}
%\date{\today}

\selectlanguage{english}
\begin{abstract}
  N. Hitchin recently introduced the notion of folded hyperK\"ahler
  metrics, in relation with $SL(\infty,\setR)$ Higgs bundles.

  We provide a construction of such metrics, and prove the
  local existence of the Hitchin component for $SL(\infty,\setR)$.
\end{abstract}

\maketitle

\selectlanguage{francais}

\section*{Introduction}
Soit $M^4$ une vari\'et\'e orient\'ee de dimension 4. Une m\'etrique
hyperk\"ahl\'erienne sur $M$ peut \^etre vue comme la donn\'ee de trois formes
symplectiques, $\omega_a$, telles que
\begin{equation}
 \omega_a \land \omega_b = \delta_{ab} v,\label{eq:1}
\end{equation}
o\`u $v$ est une forme volume sur $M$. Il existe alors une m\'etrique $g$
et trois structures complexes $J_a$ sur $M$ par rapport
auxquelles $g$ est k\"ahl\'erienne, avec formes de K\"ahler $\omega_a$.

N. Hitchin \cite{Hit14} a introduit la notion de m\'etrique hyperk\"ahl\'erienne
\emph{pli\'ee} (folded) : la 4-forme $v$ n'est plus une forme volume,
mais peut s'annuler transversalement sur une sous-vari\'et\'e $X^3\subset M^4$ ;
sur $M\setminus X$ on obtient alors une m\'etrique hyperk\"ahl\'erienne, positive ou
n\'egative suivant les composantes connexes. L'exemple standard est
le fibr\'e en 2-sph\`eres d'une surface hyperbolique $\Sigma$,
$$ M = T^*\Sigma \cup \Sigma, $$
avec $X$ le fibr\'e unitaire en cercles de $T^*\Sigma$ ; la m\'etrique est
l'analogue non compact de la m\'etrique de Eguchi-Hanson sur
$T^*P^1$. Dans ce cas, les formes $\omega_2$ et $\omega_3$ se restreignent en
un couple g\'en\'erique de 2-formes ferm\'ees sur $X$, alors que $\omega_1$
s'annule. Il y a une involution $\iota$ qui \'echange les deux c\^ot\'es en
fixant $X$, et
\begin{equation}
 \iota^*g = -g, \quad \iota^*\omega_1 = -\omega_1, \quad \iota^*\omega_2=\omega_2, \quad
\iota^*\omega_3=\omega_3.\label{eq:2}
\end{equation}

Il y a deux constructions de m\'etriques hyperk\"ahl\'eriennes pli\'ees \cite{Hit14} :
\begin{itemize}
\item une construction locale, qui \`a partir d'un couple g\'en\'erique
  $(\omega_2,\omega_3)$ de 2-formes ferm\'ees analytiques r\'eelles sur $X$, produit
  une m\'etrique hyperk\"ahl\'erienne pli\'ee dans un voisinage ; cette
  m\'etrique poss\`ede une involution $\iota$ comme ci-dessus (une autre
  approche pour ce r\'esultat est propos\'ee dans la section
  \ref{sec:la-geometrie-au}, voir th\'eor\`eme \ref{theo:local} ; cette
  approche aboutit aussi \`a un \'enonc\'e d'unicit\'e qui implique
  l'existence locale de l'involution $\iota$) ;
\item une construction globale \`a partir de solutions des \'equations
  d'auto-dualit\'e de Hitchin \cite{Hit87} pour des $SL(\infty,\setR)$-fibr\'es de
  Higgs sur $\Sigma$ ; si on ne sait pas produire en g\'en\'eral de telle
  solution, une famille de dimension finie vient du plongement
  $SL(2,\setR)\subset SL(\infty,\setR)$ ; cette famille contient le mod\`ele standard,
  induit par le fibr\'e de Higgs correspondant \`a la repr\'esentation
  fuchsienne de $\pi_1(\Sigma)$ dans $SL(2,\setR)$.
\end{itemize}

La construction \`a partir de $SL(\infty,\setR)$-fibr\'es de Higgs sugg\`ere que les
m\'etriques hyperk\"ahl\'eriennes pli\'ees doivent venir dans des familles de
dimension infinie. Le but de cet article est de confirmer cette
intuition et de d\'ecrire l'espace des d\'eformations. Il est aussi de
montrer l'existence de la composante de Hitchin pour $SL(\infty,\setR)$, qui
correspond aux m\'etriques hyperk\"ahl\'eriennes pli\'ees munies d'une
projection holomorphe sur la surface $\Sigma$.

Le premier r\'esultat de cet article est formul\'e dans le cadre o\`u $M$ est r\'eunion de deux domaines ferm\'es, d\'elimit\'es par $X$ :
$$ M=M_0 \cup M_1, \quad M_0 \cap M_1=X, $$
\'echang\'es par l'involution $\iota$. Si la forme symplectique holomorphe
$\omega^c=\omega_2+i\omega_3$ d'une m\'etrique hyperk\"ahl\'erienne pli\'ee le long de $X$
n'est plus symplectique le long de $X$, en revanche, sur le quotient
par l'involution $\iota$,
$$ M_s := M/\iota , $$
elle d\'efinit une forme symplectique holomorphe jusqu'au bord. La structure diff\'erentielle de $M_s$ au bord $X$ diff\`ere de celle de $M_0$ : si $x$ est une \'equation lisse de $X$ dans $M_0$, alors $x^2$ est une \'equation lisse du bord dans $M_s$ (d'o\`u le \emph{pli}).

\begin{theo}\label{theo:principal}
  Soit une m\'etrique hyperk\"ahl\'erienne $g$ pli\'ee sur $M$. Alors :
  \begin{enumerate}
  \item toutes les d\'eformations infinit\'esimales de m\'etriques
    hyperk\"ahl\'eriennes pli\'ees s'int\`egrent en des m\'etriques
    hyperk\"ahl\'eriennes pli\'ees ;
  \item toute d\'eformation infinit\'esimale de la vari\'et\'e
    holomorphe symplectique \`a bord $M_s$ donne lieu \`a une d\'eformation
    infinit\'esimale de m\'etrique hyperk\"ahl\'erienne pli\'ee, quitte \`a
    modifier $M_s$ par un diff\'eomorphisme infinit\'esimal ne pr\'eservant
    pas n\'ecessairement le bord $X$.
  \end{enumerate}
\end{theo}
Comme il y a beaucoup de d\'eformations infinit\'esimales holomorphes symplectiques, le th\'eor\`eme fournit bien la construction de m\'etriques hyperk\"ahl\'eriennes pli\'ees.

Il peut sembler curieux de faire agir les diff\'eomorphismes ne
pr\'eservant pas le bord, mais cela a un sens pour les diff\'eomorphismes
infinit\'esimaux : le champ de vecteurs n'est pas n\'ecessairement tangent
au bord. Cette description sugg\`ere que les d\'eformations de m\'etriques
hyperk\"ahl\'eriennes pli\'ees sont li\'ees aux d\'eformations holomorphes
symplectiques \`a fronti\`ere libre.

Pr\'ecisons la question sous-jacente : \'epaississons un peu $M_s$,
c'est-\`a-dire supposons que $$M_s \subset N,$$ o\`u $N$ est une vari\'et\'e
holomorphe symplectique sans bord (dans le cas mod\`ele, un voisinage
ouvert du fibr\'e en disques dans $T^*\Sigma$), et fixons $\zeta_1\in H^2(N,X)$ la
\guillemotleft~classe de K\"ahler pli\'ee~\guillemotright. Consid\'erons une d\'eformation holomorphe
symplectique $N'$ de $N$. Pour chaque d\'eformation $X'\subset N'$ du bord
$X$, appelons $D'$ le domaine de $N'$ d\'elimit\'e par $X'$, et fixons une
forme de K\"ahler $\omega_1$ sur $D'$, pli\'ee sur $X'$, et dans la classe
$\zeta_1$. Notons $\omega^c=\omega_2+i\omega_3$ la forme holomorphe symplectique de $N'$.

\begin{ques}[Probl\`eme de Monge-Amp\`ere \`a fronti\`ere libre]
  Trouver $(X',f)$ tel que $(\omega_1+i\partial\db f)^2=\frac12 \omega^c\land \overline{\omega^c}$, o\`u $f$ est une fonction sur $D'$, et $f=O(y)$ pr\`es de $X'$ (o\`u $y$ est une \'equation de $X$ dans $N'$).
\end{ques}

Une telle solution du probl\`eme de Monge-Amp\`ere donnerait une m\'etrique
hyperk\"ahl\'erienne pli\'ee sur le domaine de $N'$ d\'elimit\'e par $X'$. Le
point important ici est qu'il est n\'ecessaire de pouvoir d\'eplacer $X'$
pour r\'esoudre l'\'equation. La question pr\'esente des analogies avec les
questions de S.~Donaldson \`a fronti\`ere libre \cite{Don10}. Voir la fin
de la section \ref{sec:les-deform-infin} pour une interpr\'etation en
termes de la complexification du groupe des symplectomorphismes
induisant un contactomorphisme au bord.

La composante de Hitchin pour le groupe $SL(\infty,\setR)$ s'interpr\`ete
\cite{Hit14} comme un espace de m\'etriques hyperk\"ahl\'eriennes pli\'ees
avec projection holomorphe sur la surface $\Sigma$, c'est-\`a-dire l'ensemble
des domaines de $T^*\Sigma$ portant une m\'etrique hyperk\"ahl\'erienne
pli\'ee. Cela correspond \`a r\'esoudre la question ci-avant pour des
domaines de $T^*\Sigma$. Nous d\'emontrons l'existence locale de cette
composante :
\begin{theo}\label{th:2}
  Au voisinage de la m\'etrique hyperk\"ahl\'erienne pli\'ee standard sur le
  fibr\'e en disques de $T^*\Sigma$, la composante de Hitchin pour le groupe
  $SL(\infty,\setR)$ est une vari\'et\'e param\'etr\'ee par
  $ \oplus_{n\geq 2} H^0(\Sigma,K^n) $.
\end{theo}
Voir le th\'eor\`eme \ref{theo:comp-H} pour l'\'enonc\'e technique pr\'ecis :
l'espace de diff\'erentielles holomorphes $\oplus_{n\geq 2} H^0(\Sigma,K^n)$ est
interpr\'et\'e comme un espace de fonctions CR holomorphes sur le bord du
fibr\'e en disques, et une certaine r\'egularit\'e dans les espaces de
Folland-Stein est n\'ecessaire. Le lien entre ces fonctions et les d\'eformations du domaine correspondant \`a la composante de Hitchin est le suivant : ces d\'eformations correspondent infinit\'esimalement au d\'eplacement du bord du fibr\'e en disques par un champ de vecteurs $fw\partial_w$, o\`u $f$ est une fonction CR holomorphe sur le bord et $w\partial_w$ est le vecteur de dilatation dans $T^*\Sigma$.

Le th\'eor\`eme confirme l'intuition que la composante de Hitchin pour
$SL(\infty,\setR)$ devrait \^etre une sorte de limite des composantes de Hitchin
pour les groupes $SL(k,\setR)$, lesquelles sont param\'etr\'ees par des sommes
finies d'espaces de diff\'erentielles holomorphes. En outre, la
param\'etrisation dans le th\'eor\`eme \ref{th:2} peut \^etre choisie pour
\^etre une section d'un analogue de la fibration de Hitchin, voir
remarque \ref{rema:section-H}.

Les sections \ref{sec:la-geometrie-au} et \ref{sec:lesp-des-metr} sont
consacr\'ees \`a la description de la g\'eom\'etrie au bord et \`a la mise en
forme comme un probl\`eme non lin\'eaire sur les diff\'erentielles de trois
1-formes ; l'analyse des d\'eformations hyperk\"ahl\'eriennes de cette
mani\`ere n'est pas nouvelle, ce qui compte ici est de d\'eterminer les
conditions au bord correspondant \`a la g\'eom\'etrie (un travail en cours
de J.~Fine, J.~Lotay et M.~Singer analyse le cas d'un bord
standard). Les espaces fonctionnels ad\'equats, repris de \cite{Biq05},
sont introduits dans la section \ref{sec:espaces-fonctionnels}, la
lin\'earisation du probl\`eme est analys\'ee section
\ref{sec:constr-dun-inverse}, et section \ref{sec:les-deform-infin}
les d\'eformations infinit\'esimales sont comprises en termes de la
g\'eom\'etrie holomorphe symplectique. Dans la section
\ref{sec:lespace-des-modules-1}, on d\'etermine les d\'eformations
infinit\'esimales correspondant aux $SL(\infty,\setR)$-fibr\'es de Higgs sur la
surface de Riemann $\Sigma$ : elles sont param\'etr\'ees par les
diff\'erentielles holomorphes de tous degr\'es (au moins quadratiques). Un
probl\`eme technique se pose alors pour parvenir au th\'eor\`eme \ref{th:2}
: la param\'etrisation du d\'eplacement du bord du domaine
holomorphe symplectique par une fonction (donnant le d\'eplacement
radial) se fait a priori avec perte de d\'eriv\'ees. Cette question est
contourn\'ee par la section \ref{sec:param-des-doma} qui propose une
param\'etrisation de tous les domaines du cotangent en termes de
fibrations (non holomorphes) par des disques holomorphes, suivant des
id\'ees remontant \`a Burns, Epstein, Lempert et Bland dans les ann\'ees 90,
notre approche ici \'etant bas\'ee sur \cite{Biq02}. Cela permet de
d\'eduire le th\'eor\`eme \ref{th:2} section
\ref{sec:la-composante-de}. Finalement, l'asymptotique au bord des
m\'etriques n\'ecessite le d\'eveloppement d'une analyse, report\'ee jusqu'\`a
la section \ref{sec:analyse}.

Mes remerciements vont \`a N.~Hitchin, pour les nombreux \'echanges qui
ont permis l'existence de cet article. Je remercie aussi C.~Guillarmou
pour d'utiles discussions sur le laplacien pli\'e au d\'ebut de ce travail.

\section{La g\'eom\'etrie au bord et son mod\`ele}
\label{sec:la-geometrie-au}

On commence par pr\'eciser le comportement au bord d'une m\'etrique
hyperk\"ahl\'erienne pli\'ee \cite{Hit14}. Nous avons un triplet
$(\omega_1,\omega_2,\omega_3)$ de 2-formes sur une vari\'et\'e $M$, qui en dehors d'une
hypersurface $X$ (le \guillemotleft~pli~\guillemotright) donne une m\'etrique hyperk\"ahl\'erienne
(d\'efinie positive ou d\'efinie n\'egative). Soit $i:X\hookrightarrow M$ l'injection.
On a 
\begin{equation}
i^*\omega_1=0,\label{eq:3}
\end{equation}
alors que les formes $\omega_2$ et $\omega_3$, restreintes \`a $X$, ont chacune un
noyau de dimension 1, dont la somme est une distribution de contact:
\begin{equation}
 H = \ker i^*\omega_2 \oplus \ker i^*\omega_3.\label{eq:4}
\end{equation}
Dans cette situation, R.~Bryant \cite{Bry04} a montr\'e l'existence d'une
unique base $(\theta^1,\theta^2,\theta^3)$ de 1-formes sur $X$, telle que
\begin{gather}
    i^*\omega_2 = - \theta^1 \land \theta^3, \quad i^*\omega_3 = \theta^1 \land \theta^2, \label{eq:5} \\
    d\theta^1 = \theta^2 \land \theta^3. \label{eq:6}
\end{gather}
Bien s\^ur, la forme $\theta^1$ est une forme de contact, et les formes $\theta^2$
et $\theta^3$ sont horizontales, c'est-\`a-dire qu'elles s'annulent sur le
champ de Reeb $X_1$ (c'est donc le premier vecteur de la base duale
$(X_1,X_2,X_3)$). 

Alors il existe une \'equation $x$ de $X\subset M$, dont la diff\'erentielle le
long de $X$ est bien d\'etermin\'ee, et telle que
\begin{equation}
\begin{split}
  \omega_1|_X &= dx \land \theta^1 + x \theta^2 \land \theta^3, \\
  \omega_2|_X &= x dx \land \theta^2 - \theta^1 \land \theta^3, \\
  \omega_3|_X &= x dx \land \theta^3 + \theta^1 \land \theta^2. 
\end{split}\label{eq:7}
\end{equation}
Ce comportement sera extrait du r\'esultat suivant, qui donne
l'existence locale et l'unicit\'e de la m\'etrique hyperk\"ahl\'erienne pli\'ee
:

\begin{theo}\label{theo:local}
  \'Etant donn\'e $(X^3,\beta_2,\beta_3)$ analytique r\'eel, o\`u $\beta_2$ et $\beta_3$ sont
  des 2-formes ferm\'ees sur $X$ dont les noyaux engendrent une
  distribution de contact, il existe sur un petit voisinage $(-\epsilon,\epsilon)\times X$
  une unique m\'etrique hyperk\"ahl\'erienne pli\'ee telle que $i^*\omega_2=\beta_2$ et
  $i^*\omega_3=\beta_3$. Cette m\'etrique satisfait la parit\'e (\ref{eq:2}).
\end{theo}
L'existence est d\'emontr\'ee par une construction twistorielle
\cite[\S7]{Hit14}. La d\'emonstration que nous donnons ici simplifie
cette preuve et aboutit directement au r\'esultat d'unicit\'e.
\begin{proof}
  On utilise le formalisme d'Ashtekar \cite{AshJacSmo88} : une
  solution du syst\`eme des \'equations de Nahm pour des champs de
  vecteurs $V_1$, $V_2$, $V_3$ sur $X$, d\'ependant de $x$, et
  pr\'eservant une forme volume fixe $\upsilon$ sur $X$,
  \begin{equation}
    \label{eq:8}
    \begin{split}
      \frac{dV_1}{dx}+[V_2,V_3]&=0, \\ \frac{dV_2}{dx}+[V_3,V_1]&=0,
      \\ \frac{dV_3}{dx}+[V_1,V_2]&=0,
    \end{split}
  \end{equation}
  produit, en posant $V_0=\frac \partial{\partial x}$, une m\'etrique hyperk\"ahl\'erienne
  d\'efinie par
  \begin{equation}
    \label{eq:9}
    g(V_i,V_j) = \upsilon(V_1,V_2,V_3) \delta_{ij}.
  \end{equation}
  R\'eciproquement, si $g$ est une m\'etrique hyperk\"ahl\'erienne et $x$ une
  fonction harmonique, alors, en posant $dx\land \upsilon=|dx|_g^2 \vol^g$ et $V_a=J_a\frac
  \partial{\partial x}$ pour $a=1\dots 3$, on r\'ecup\`ere une solution du syst\`eme
  (\ref{eq:8}).

  Appliquons cela dans notre situation : partant de $(X,\beta_2,\beta_3)$, on
  prend la base de 1-formes $(\theta^1,\theta^2,\theta^3)$ satisfaisant $d\theta^1=\theta^2\land
  \theta^3$, $\beta_2=-\theta^1\land \theta^3$ et $\beta_3=\theta^1\land \theta^2$, et $(X_1,X_2,X_3)$ la base
  associ\'ee de champs de vecteurs. Alors les conditions $d\beta_2=d\beta_3=0$
  se traduisent par le fait que $X_2$ et $X_3$ pr\'eservent la forme
  volume $\upsilon=\theta^1\land \theta^2\land \theta^3$. On r\'esout alors le syst\`eme (\ref{eq:8})
  avec les conditions initiales 
  \begin{equation}
    \label{eq:10}
    V_1(0) = 0, \quad V_2(0) = X_2, \quad V_3(0) = X_3.
  \end{equation}
  Pour des donn\'ees analytiques r\'eelles, le th\'eor\`eme de
  Cauchy-Kowalevski produit une unique solution d\'efinie pour $x$
  petit.

  On observera que $(-V_1(-x),V_2(-x),V_3(-x))$ est encore solution
  avec les m\^emes conditions initiales, donc $V_1$ est paire, et $V_2$,
  $V_3$ impaires, ce qui implique l'invariance (\ref{eq:2}) sous
  l'involution $\iota(x)=-x$ pour la solution. En outre, puisque
  $X_1=-[X_2,X_3]$, on a
  \begin{equation}
    \label{eq:11}
    V_1(x) = xX_1 + O(x^3).
  \end{equation}
  On d\'eduit le comportement de la m\'etrique (impaire, positive pour
  $x>0$, n\'egative pour $x<0$) :
  \begin{multline}
    \label{eq:12}
    g = x(dx^2+(\theta^2)^2+(\theta^3)^2) + x^{-1}(\theta^1)^2 \\ + O(x^3) G(dx,x^{-1}\theta^1,\theta^2,\theta^3),
  \end{multline}
  et celui des trois formes de K\"ahler donn\'e dans (\ref{eq:7}).
  Ici $G((e^i))=\sum G_{ij}e^ie^j$ est un 2-tenseur sym\'etrique dont les coefficients $G_{ij}$ sont lisses.

  R\'eciproquement, \'etant donn\'ee une m\'etrique hyperk\"ahl\'erienne,
  analytique r\'eelle, avec le comportement (\ref{eq:12}), on calcule
  son laplacien
  \begin{equation}
    \label{eq:13}
    \Delta = -x^{-1}(\partial_x^2+x^2X_1^2+X_2^2+X_3^2)+\cdots 
  \end{equation}
  Il en r\'esulte imm\'ediatement qu'on peut r\'esoudre $\Delta y=0$ dans un
  voisinage de $X$ avec $y=x + O(x^2)$ ; cette solution, unique,
  permet de reconstruire les champs $V_a$. L'unicit\'e s'en d\'eduit.
\end{proof}

Il est int\'eressant de noter qu'existe un cas o\`u les formules
(\ref{eq:7}) sont exactes globalement et pas seulement sur $X$ : si
$X$ est le groupe de Heisenberg, muni de sa base invariante de
1-formes telle que
\begin{equation}
  \label{eq:14}
  d\theta^1 = \theta^2\land \theta^3, \quad d\theta^2=d\theta^3=0,
\end{equation}
alors $(V_1,V_2,V_3)(x)=(xX_1,X_2,X_3)$ est une solution exacte de
(\ref{eq:8}), donc les formules (\ref{eq:7}) d\'efinissent des 2-formes
ferm\'ees satisfaisant le syst\`eme (\ref{eq:1}) sur $M=\setR\times X$, et la
m\'etrique hyperk\"ahl\'erienne pli\'ee $g$ est explicit\'ee par
\begin{equation}
  \label{eq:15}
  g_0 = x \big(dx^2+(\theta^2)^2+(\theta^3)^2\big) + x^{-1} (\theta^1)^2.
\end{equation}
(On peut le voir aussi par application de l'ansatz de Gibbons-Hawking).

Le cas du groupe de Heisenberg est le mod\`ele \guillemotleft{} plat \guillemotright{} de la g\'eom\'etrie
que nous \'etudions, au sens suivant. Prenons des coordonn\'ees
$(x^1,x^2,x^3)$, de sorte que
$$ \theta^1 = dx^1 + x^2 dx^3, \quad \theta^2=dx^2, \quad \theta^3=dx^3. $$
Dans le cas d'une m\'etrique hyperk\"ahl\'erienne pli\'ee g\'en\'erale $g$, soit
un point $p\in X$, choisissons gr\^ace au lemme de Darboux des coordonn\'ees locales
$(x^i)$ sur $X$ en $p$ de sorte que
$$ \theta^1 = dx^1 + x^2 dx^3, \quad \theta^2(p) = dx^2, \quad \theta^3(p) = dx^3. $$
Consid\'erons les dilatations inhomog\`enes
$$ h_t(x,x^1,x^2,x^3) = (tx,t^2x^1,tx^2,tx^3). $$
Alors la m\'etrique mod\`ele (\ref{eq:15}) satisfait $h_t^*g_0=t^3g_0$, et
plus g\'en\'eralement, \`a partir de (\ref{eq:7}), quand $t\to0$, on voit que
les $t^{-3}h_t^*\omega_a$ convergent vers les 2-formes du mod\`ele, et en particulier
\begin{equation}
 \lim_{t\to0} t^{-3}h_t^* g = g_0.\label{eq:16}
\end{equation}

Il y a une analogie claire avec la g\'eom\'etrie hyperbolique complexe et
les m\'etriques asymptotiquement hyperboliques complexes \cite{Biq00},
mais qui n'est qu'une analogie : en effet, la m\'etrique $x^{-3}g_0$,
invariante par les dilatations $h_t$, n'est pas hyperbolique complexe. Elle est n\'eanmoins quasi-isom\'etrique \`a la m\'etrique hyperbolique complexe.

\section{L'espace des m\'etriques hyperk\"ahl\'eriennes pli\'ees}
\label{sec:lesp-des-metr}

Nous consid\'erons \`a pr\'esent les d\'eformations d'une m\'etrique
hyperk\"ahl\'erienne pli\'ee $g_0$ sur $(M,X)$. Puisque deux structures de
contact proches sont diff\'eomorphes, on peut supposer que la
distribution de contact $H$ induite sur $X$ par (\ref{eq:4}) reste
fixe. Les formes $i^*\omega_2$ et $i^*\omega_3$ sont alors n\'ecessairement des
2-formes verticales sur $X$ (et $i^*\omega_1=0$). Enfin nous consid\'ererons
les d\'eformations en fixant les classes de cohomologie des formes $\omega_a$
: notons $\zeta_1$ la classe de $\omega_1$ dans $H^2(M_0,X)$, et $\zeta_2$, $\zeta_3$
les classes de $\omega_2$ et $\omega_3$ dans $H^2(M_0)$.

Cela nous am\`ene \`a consid\'erer l'espace $\cQ$ des triplets $(\omega_a)$ de
2-formes ferm\'ees sur $M_0$, de classes de cohomologie $(\zeta_a)$ dans
$H^2(M_0,X)$ ou $H^2(M_0)$ respectivement, tels que
\begin{equation}
 i^*\omega_1=0, \quad i^*\omega_2,i^*\omega_3 \text{ verticales.}\label{eq:17}
\end{equation}

L'espace des m\'etriques hyperk\"ahl\'eriennes pli\'ees sur $M_0$ est alors
\begin{equation}
 \cM = \{ (\omega_a)\in \cQ, \omega_a\land \omega_b = \delta_{ab} \omega_1^2 \} .\label{eq:18}
\end{equation}
Par le th\'eor\`eme \ref{theo:local}, de telles m\'etriques satisfont
n\'ecessairement la parit\'e (\ref{eq:2}) pr\`es du bord $X$ (pour une
certain choix de $x$), ce qui implique qu'on peut les prolonger par
doublement en des m\'etriques hyperk\"ahl\'eriennes pli\'ees sur $M$
entier. (Le th\'eor\`eme \ref{theo:local} n'est valable que pour des
donn\'ees analytiques r\'eelles, mais si celles-ci sont seulement $C^\infty$, il
donne n\'eanmoins le m\^eme r\'esultat sur les germes en $X$, ce qui suffit
pour le prolongement par doublement).

Aussi raisonnerons-nous uniquement sur la vari\'et\'e \`a bord $M_0$.

Le but de cet article est de comprendre l'espace $\cM$. On peut ainsi d\'ecrire $\cQ$ \`a partir de
\begin{equation}
  \cT = \{ (\alpha_a)\in \Omega^1(M_0), i^*\alpha_1=0, i^*d\alpha_2, i^*d\alpha_3 \text{
    verticales} \}. \label{eq:19}
\end{equation}
Les conditions sur les 1-formes sont \'ecrites de sorte que $(\alpha_a)\in \cT$
implique $(\omega_a+d\alpha_a)\in \cQ$, et tout \'el\'ement de $\cQ$ s'\'ecrit de cette
mani\`ere. L'espace $\cM$ se d\'ecrit comme l'image par $d$ de
$P^{-1}(0)$, pour
\begin{equation}
  \label{eq:20}
  P\big((\alpha_a)\big) = \big((\omega_a+d\alpha_a) \land (\omega_b+d\alpha_b)\big)_0,
\end{equation}
o\`u l'indice $0$ d\'enote la partie sans trace ; on a donc d\'efini un op\'erateur
\begin{equation}
  \label{eq:21}
  P : \cT \longrightarrow \Sym_0^2(\setR^3)\otimes\Omega^4.
\end{equation}

Cet op\'erateur, et sa lin\'earisation, interviennent classiquement dans
les probl\`emes d'autodualit\'e, voir par exemple \cite{Biq05} dans un
contexte proche.

Observons qu'il y a une contrainte sur l'image de $P$ : en effet, les
nombres $\zeta_1\cup \zeta_2, \zeta_1\cup \zeta_3\in H^4(M_0,X)=\setR$ sont repr\'esent\'es par
\begin{equation}
  \label{eq:22}
  \zeta_1\cup \zeta_b = \int_{M_0} (\omega_1+d\alpha_1)\land (\omega_b+d\alpha_b), \quad b=2,3.
\end{equation}

L'analyse de l'op\'erateur $P$ requiert d'introduire les espaces
fonctionnels ad\'equats, ce que nous faisons maintenant.

\section{Espaces fonctionnels}
\label{sec:espaces-fonctionnels}

Les espaces fonctionnels sur le bord, adapt\'es \`a la g\'eom\'etrie de
contact, sont les espaces de Folland-Stein $\FS^k$ \cite{FolSte74} : une fonction $f$
sur $X$ est dans l'espace $\FS^k$ si elle a $k$ d\'eriv\'ees horizontales
dans $L^2$, c'est-\`a-dire $$X_2^{j_2}X_3^{j_3}f\in L^2(X) \text{ d\`es que }
j_2+j_3\leq k.$$ Une version fractionnaire est d\'efinie en consid\'erant
l'op\'erateur hypoelliptique $\square=-(X_2^2+X_3^2)$, dont la
d\'ecomposition spectrale permet de d\'efinir la norme
$\|f\|_{\FS^k} = \|(1+\square^{\frac k2})f\|_{L^2}$.

Une caract\'eristique, presque une d\'efinition des m\'etriques pli\'ees, est
que la forme volume s'annule simplement sur $X$ : pr\`es de $X$, on a
$$ \vol^{g_0} \sim x dx\land \theta^1\land \theta^2\land \theta^3. $$
Nous consid\'erons l'espace $L^2$ par rapport \`a cette forme volume, et
un espace de fonctions $L^2$ \`a poids par
$$ L^2_\delta=x^{\delta+1}L^2. $$
La d\'efinition est faite pour que $x^{\delta'}\in L^2_\delta$ d\`es que $\delta'>\delta$.

Rappelons que sur $X$ nous disposons d'un rep\`ere $(X_1,X_2,X_3)$ dual
\`a $(\theta^1,\theta^2,\theta^3)$, et nous pouvons identifier un voisinage de $X$ dans
$M_0$ \`a $[0,\epsilon)\times X$, avec premi\`ere coordonn\'ee $x$, ce qui nous permet
d'ajouter le champ de vecteurs $\partial_x$. On consid\`ere alors, pour $s\in \setN$,
l'espace de Sobolev \`a poids d\'efini par la norme
$$ \|f\|^2_{H^s_\delta} = \sum_{|j|:=j_0+\cdots+j_3\leq s}
\|\partial_x^{j_0}(xX_1)^{j_1}X_2^{j_2}X_3^{j_3}f\|^2_{L^2_{\delta-|j|}}. $$ La
norme est ind\'ependante de l'ordre dans lequel on \'ecrit les champs de
vecteurs, car
$$ [\partial_x,xX_1] = X_1 = - [X_2,X_3] . $$
Par commodit\'e d'\'ecriture, on notera $D$ toute d\'erivation parmi $\partial_x$,
$xX_1$, $X_2$ et $X_3$ et $D^j$ toute composition d'ordre $j$ de ces
d\'erivations. La norme de Sobolev pr\'ec\'edente s'\'ecrit ainsi
$$ \|f\|^2_{H^s_\delta} = \sum_{j\leq s} \|D^jf\|^2_{L^2_{\delta-j}}. $$

Une autre interpr\'etation des espaces de Sobolev \`a poids s'obtient en consid\'erant la m\'etrique $\frac{g_0}{x^3}$, quasi-isom\'etrique \`a une m\'etrique asymptotiquement hyperbolique complexe, donc on dispose d'espaces de Sobolev $H^s_{\frac{g_0}{x^3}}$ obtenus en sommant les carr\'es des normes $H^s$ ordinaires sur un recouvrement localement fini par des boules. Le lien avec les espaces de Sobolev d\'efinis pr\'ec\'edemment est
\begin{equation}
 H^s_\delta = x^{\delta-2} H^s_{\frac{g_0}{x^3}}.\label{eq:23}
\end{equation}
Cette relation permet en particulier d'\'etendre la d\'efinition de $H^s_\delta$ \`a des $s$ fractionnaires.

Enfin, on utilisera aussi une variante o\`u $\ell$ d\'eriv\'ees (o\`u $\ell\in \setN$) le long de $X$
sont mieux contr\^ol\'ees : si $s\geq \ell$, on note $H^{s,\ell}_\delta$ l'espace des fonctions
$f\in H^s_\delta$, telles que pour $|j|:=j_2+j_3\leq \ell$ on ait
$$ X_2^{j_2}X_3^{j_3}f \in H^{s-|j|}_\delta . $$
Ce type d'espace est utilis\'e dans \cite{Biq05}, d'o\`u nous extrayons le
lemme d'extension suivant :
\begin{lemm}\label{lem:extension}
  1. Soit $\delta\in (0,1)$. Si une fonction $f$ satisfait $Df\in L^2_{-1+\delta}$, alors $f$
  admet une valeur au bord $f|_X\in \FS^\delta$, et $f-(f|_X)\in  L^2_\delta$.

  R\'eciproquement, il existe un op\'erateur d'extension, $E_0$, qui \`a
  $f_0\in \FS^{\ell+\delta}(X)$ associe une extension $f=E_0(f_0)$ sur $M_0$
  telle que $Df\in H^{\infty;\ell}_{-1+\delta}$.

  2. Plus g\'en\'eralement, il existe des op\'erateurs d'extension $E_k$,
  associant \`a un d\'eveloppement $f_0+xf_1+\cdots +x^kf_k$, o\`u $f_j\in
  \FS^{k-j+\ell+\delta}(X)$, une extension $f$ sur $M_0$, telle que
  \begin{enumerate}
  \item pour $j\leq k$ on a $D^j(f-\sum_0^jx^if_i)\in L^2_{k-j+\delta}$ ;
  \item $D^{k+1}f\in H^{\infty;\ell}_{-1+\delta}$.
  \end{enumerate}
\end{lemm}
\begin{proof}
  Ce sont les lemmes 2.5 et 2.7 dans \cite{Biq05}, qui s'appliquent car on a vu que la m\'etrique $\frac{g_0}{x^3}$ est quasi-isom\'etrique aux m\'etriques asymptotiquement hyperboliques complexes, utilis\'ees dans \cite{Biq05}. Ils n'y sont \'enonc\'es que pour une seule fonction $f_0$, mais en l'appliquant \`a chaque $f_j$ on d\'eduit l'\'enonc\'e \'ecrit ici.
\end{proof}
On peut \'etendre l\'eg\`erement le lemme \ref{lem:extension} de la mani\`ere
suivante : si on a seulement $f\in H^{\frac12}_\delta$ alors les
restrictions aux tranches $f|_{\{x\}\times X}$, bien d\'efinies dans $L^2$,
sont contr\^ol\'ees par la norme $H^{\frac12}$ de mani\`ere uniforme dans
les boules de la m\'etrique $\frac{g_0}{x^3}$, et il en r\'esulte
que $x^{-\delta}f|_{\{x\}\times X}\to0$ dans $L^2(\theta^1\theta^2\theta^3)$ quand $x$ tend vers
$0$. Donc, pour $s\geq \frac12$, l'espace 
$$\cH^s_\delta=\FS^\delta(X)\oplus H^s_\delta(M_0), $$
constitu\'e des fonctions $f$ qui se d\'ecomposent pr\`es du bord en
\begin{equation}
 f = E_0(f_0) + f_1, \quad f_0\in \FS^\delta(X), f_1\in H^s_\delta(M_0),\label{eq:24}
\end{equation}
a un sens (la projection sur $\FS^\delta(X)$ \'etant la valeur au bord).
Si $s\geq 1$, c'est exactement l'espace des fonctions $f$ sur $M_0$
telles que $Df\in H^{s-1}_{-1+\delta}(M_0)$.

De mani\`ere analogue est d\'efini, pour $s\geq \ell+\frac12$, l'espace
$$ \cH^{s,\ell}_\delta := \FS^{\ell+\delta}(X) \oplus H^{s,\ell}_\delta(M_0), $$
constitu\'e des fonctions $f$ avec une d\'ecomposition (\ref{eq:24}) avec
$f_0\in \FS^{\ell+\delta}(X)$ et $f_1\in H^{s,\ell}_\delta(M_0)$. Si $s\geq \ell+1$, c'est
l'espace des fonctions $f$ sur $M_0$ telles que $Df\in H^{s-1,\ell}_{-1+\delta}(M_0)$.

Enfin, au lieu d'un seul terme au bord, on peut d\'efinir des espaces o\`u
les fonctions disposent d'un d\'eveloppement d'ordre $k$ en $x$ et d'un
reste d'ordre $k+\delta$ : pour $s\geq \ell+\frac12$,
$$ \cH^{s,\ell}_{k,\delta} := \oplus_0^k \FS^{\ell+k-j+\delta}(X) \oplus H^{s,\ell}_{k+\delta}(M_0). $$
Si $s$ est assez grand, $\cH^{s,\ell}_{k,\delta}$ est constitu\'e des $f$ telles que $D^{k+1}f\in H^{s-k-1,\ell}_{-1+\delta}(M_0)$.

Disons tout de suite qu'on choisira dor\'enavant des valeurs 
$$ \ell\gg 0, \quad \delta=\frac12, \quad s\geq \ell+\frac12.$$
Le choix du poids $\delta=\frac12$ rend transparent le rapport aux espaces
$L^2$ ordinaires, mais nous continuerons d'utiliser $\delta$ car tous les
\'enonc\'es sont valables d\`es que $\delta\in (0,1)$, voir remarque \ref{rem:delta}.
Enfin, dans la section \ref{sec:param-des-doma} on utilisera
$s=\ell+\frac12$ donc la seule vraie libert\'e est sur $\ell$.

Finalement, ces choix permettent de plonger contin\^ument
$$\cH^{s,\ell}_{k,\delta} \subset C^k,$$ et $\cH^{s,\ell}_{k,\delta}$ est une alg\`ebre.
Dans l'int\'erieur de $M_0$, les choix permettent d'obtenir
$\cH^{s,\ell}_{k,\delta}\subset C^j_{\mathrm{loc}}$ pour tout $j$ grand pr\'ealablement fix\'e, mais en
revanche les normes \`a poids ne contr\^olent pas les d\'eriv\'ees radiales
$\partial_x^jf$ au bord.

Nous pouvons maintenant d\'efinir une version Sobolev des espaces $\cT$,
$\cQ$ et $\cM$ de la mani\`ere suivante. Nous consid\'erons la base de
1-formes $(e^i)$ d\'efinie par
$$ e^0=dx, \quad e^1=x^{-1}\theta^1, \quad e^2=\theta^2, \quad e^3=\theta^3. $$
La base $(x^{\frac12} e^i)$ est orthonormale le long de $X$. Nous
d\'efinissons alors $\cQ^{s,\ell}_\delta$ comme l'espace des triplets de 2-formes
$(\omega_a)$ :
\begin{enumerate}\item \`a coefficients $x\cH^{s,\ell}_\delta$ dans la base $(e^i\land e^j)$,
  \item satisfaisant (\ref{eq:17}) ; comme $e^2\land e^3=\theta^2\land \theta^3$, la
    verticalit\'e de $i^*\omega_2$ et $i^*\omega_3$ est impliqu\'ee par la premi\`ere
    condition, et (\ref{eq:17}) se r\'eduit donc \`a $i^*\omega_1=0$.
\end{enumerate}
La premi\`ere condition implique que les 2-formes se prolongent
contin\^ument au-dessus de $X$. Puisque
$\vol^{g_0}=xdx\land\theta^1\land\theta^2\land\theta^3=x^2e^1\land e^2\land e^3\land e^4$, un tel triplet
$(\omega_a)$ satisfait
\begin{equation}
  \label{eq:25}
  \omega_a \land \omega_b \in \cH^{s,\ell}_\delta \vol^{g_0}.
\end{equation}

On d\'efinit aussi l'espace $\cT^{s,\ell}_\delta$ des triplets de 1-formes $(\alpha_a)$ :
\begin{enumerate}\item \`a coefficients $\cH^{s,\ell}_{2,\delta}$ dans la base $(e^i)$,
\item dont la coordonn\'ee sur $e^1$ s'annule sur $X$ (donc il s'agit de
  formes s'\'etendant jusqu'au bord $X$ en des formes de classe $C^1$),
\item et tels que $i^*\alpha_1=0$ et $(d\alpha_a)\in \cQ^{s-1,\ell}_\delta$.
\end{enumerate}
Pr\'ecisons la derni\`ere condition : si les coefficients des $\alpha_a$ sont dans l'espace $\cH^{s,\ell}_{2,\delta}$, ceux des $d\alpha_a$ sont dans $\cH^{s-1,\ell}_{1,\delta}$ ; pour qu'ils soient en outre dans $x\cH^{s-1,\ell}_\delta$, il faut et il suffit que leur restriction \`a $X$ s'annule. Ces consid\'erations aboutissent \`a expliciter les condition sur $(\alpha_a)$ par :
\begin{lemm}\label{lem:condition-1-formes}
  Soit $(\alpha_a)$ un triplet de 1-formes, \`a coefficients $\cH^{s,\ell}_{2,\delta}$ dans la base $(e^i)$, dont le coefficient sur $e^1$ s'annule sur le bord $X$, et tel que $i^*\alpha_1=0$. Notons $\alpha_a=\alpha_{a,i}e^i$. Alors $(\alpha_a)\in \cT^{s,\ell}_\delta$ si les coefficients
$$ \partial_x\alpha_{a,2}-X_2\alpha_{a,0}, \quad \partial_x\alpha_{a,3}-X_3\alpha_{a,0}, \quad X_2\alpha_{a,3}-X_3\alpha_{a,2}+\partial_x\alpha_{a,1} $$
s'annulent le long de $X$.
\end{lemm}
\begin{proof}
  Il suffit de prendre la diff\'erentielle ext\'erieure de $\alpha_a=\alpha_{a,0}dx+\alpha_{a,1}x^{-1}\theta^1+\alpha_{a,2}\theta^2+\alpha_{a,3}\theta^3$, en tenant compte de $\alpha_{a,1}|_X=0$, sachant que $d\theta^1=\theta^2\land \theta^3$ et que $d\theta^2$ et $d\theta^3$ sont verticales. (La troisi\`eme annulation est automatique pour $\alpha_1$ puisque $i^*\alpha_1=0$).
\end{proof}

On consid\`ere alors $P$ comme un op\'erateur
\begin{equation}
  \label{eq:26}
  P : \cT^{s+1,\ell}_\delta \longrightarrow \cH^{s,\ell}_{\delta;\zeta}(\Sym_0^2\setR^3) \vol^{g_0},
\end{equation}
o\`u l'indice $\zeta$ signifie, conform\'ement \`a (\ref{eq:22}), que $v=(v_{ab})\in
\Sym_0^2\setR^3\otimes \vol^{g_0}$ satisfait
\begin{equation}
\int_{M_0} v_{12} = \zeta_1\cup \zeta_2, \quad \int_{M_0} v_{13} = \zeta_1 \cup \zeta_3.\label{eq:27}
\end{equation}

Les m\'etriques hyperk\"ahl\'eriennes dans $\cQ^{s,\ell}_\delta$ sont donc obtenues
comme l'espace
\begin{equation}
\cM^{s,\ell}_\delta = d(P^{-1}(0)) \subset \cQ^{s,\ell}_\delta.\label{eq:28}
\end{equation}
Par la construction twistorielle, toutes les m\'etriques de $\cM^{s,\ell}_\delta$ sont
n\'ecessairement lisses \`a l'int\'erieur de $M_0$, donc, \`a donn\'ee du bord
fix\'ee, les diff\'erents $\cM^{s,\ell}_\delta$ ne diff\`erent que par l'action de
diff\'eomorphismes non $C^\infty$ (aucune jauge pour l'action des
diff\'eomorphismes n'est impos\'ee dans notre construction). Par ailleurs,
varier $\ell$ permet d'avoir des donn\'ees au bord non r\'eguli\`eres, mais ce
n'est pas notre int\'er\^et principal ici.

\begin{theo}\label{th:submersion}
  L'op\'erateur $P$ est une submersion en $g_0$. Par cons\'equent, $\cM^{s,\ell}_\delta$
  est une sous-vari\'et\'e hilbertienne de $\cQ^{s,\ell}_\delta$, dont l'espace tangent
  en $g_0$ est $d(\ker d_{g_0}P)$.
\end{theo}
La seconde partie du th\'eor\`eme \ref{theo:principal} en d\'ecoule.

Le th\'eor\`eme \ref{th:submersion} est une cons\'equence de la proposition \ref{prop:inverse_a_droite}, d\'emontr\'ee dans la section suivante.

\section{Construction d'un inverse \`a droite}
\label{sec:constr-dun-inverse}

Soit $\Omega_+$ le fibr\'e des 2-formes autoduales, donc la base $(\omega_a)$
donne une trivialisation $\Omega_+=\setR^3$. Alors l'op\'erateur lin\'earis\'e
$d_{g_0}P$ s'identifie \`a la composition de $d_+$ avec la projection
$\Omega_+\otimes\Omega_+\to\Sym_0^2\Omega_+$,
\begin{equation}
  \label{eq:29}
  \partial : \cT^{s+1,\ell}_\delta \longrightarrow \cH^{s,\ell}_{\delta;0}(\Sym_0^2 \Omega_+), 
\end{equation}
o\`u l'indice $0$ marque maintenant la condition (\ref{eq:27})
lin\'earis\'ee, \`a savoir, pour $v=(v_{ab}\omega_a\otimes \omega_b)$ sym\'etrique,
\begin{equation}
  \label{eq:30}
  \int_{M_0} v_{12} \vol^{g_0} = \int_{M_0} v_{13} \vol^{g_0} = 0.
\end{equation}

On montre la surjectivit\'e de $\partial$ en consid\'erant plut\^ot le laplacien
\begin{equation}
  \label{eq:31}
  \partial\partial^*=d_+ d_+^* : \Sym_0^2 \Omega_+ \longrightarrow \Sym_0^2 \Omega_+.
\end{equation}
Puisque $\Omega_+=\setR^3$ est plat, l'op\'erateur $\partial\partial^*$ s'identifie au laplacien
scalaire $\Delta$ agissant sur chaque coefficient de la matrice
sym\'etrique. Un calcul direct donne
\begin{equation}
  \label{eq:32}
  \Delta = -x^{-1}\left(\partial_x^2+x^2X_1^2+X_2^2+X_3^2\right) + xF(\partial_x,xX_1,X_2,X_3),
\end{equation}
o\`u $F$ est un op\'erateur diff\'erentiel, impair, \`a coefficients $C^\infty$
jusqu'au bord. Le coefficient $x$ provient de la parit\'e de la m\'etrique
$g$, et, dans le cas plat, on a $F=0$.

Dans la section \ref{sec:analyse}, on montrera que le laplacien $\Delta$
sur $M_0$, consid\'er\'e sur nos espaces fonctionnels, se comporte
de mani\`ere similaire \`a un laplacien ordinaire sur une vari\'et\'e \`a
bord. A priori, on consid\`ere l'op\'erateur
$$ \Delta : x^3\cH^{s+2,\ell}_\delta \longrightarrow \cH^{s,\ell}_\delta. $$
\'Evidemment, cet op\'erateur n'est pas surjectif, car on ne peut pas
esp\'erer r\'esoudre le probl\`eme de Dirichlet, par exemple, avec une
donn\'ee de Neumann nulle aussi. On s'attend plut\^ot \`a ce qu'une solution de
$\Delta f=g$ avec $g\in \cH^{s,\ell}_\delta$ soit dans l'espace $\cH^{s+2,\ell}_{3,\delta}$,
donc avec un d\'eveloppement pr\`es de $X$ de la forme
$$ f \sim f_0 + x f_1 + x^2 f_2 + x^3 f_3 + \cdots , $$
dans lequel $f_0$ et $f_1$ sont ind\'etermin\'es, mais $f_2$ et $f_3$ sont
d\'etermin\'es formellement par $f_0$ et $f_1$ et $g=\Delta f=O(1)$, donc en
particulier 
\begin{equation}
  f_2 = -\frac 12 (X_2^2+X_3^2)f_0.\label{eq:33}
\end{equation}
(On a l'\'equation similaire sur $f_3$, avec un terme additionnel $g|_X$).

On d\'emontrera dans la section \ref{sec:analyse} (proposition
\ref{prop:laplacien-mod-poids}) que, pour $g\in \cH^{s,\ell}_\delta$, l'\'equation $\Delta f=g$
admet une unique solution $f\in \cH^{s+2,\ell}_{3,\delta}$ dans les deux cas
suivants :
\begin{itemize}
\item $f$ satisfait la condition de Dirichlet (donc $f_0=0$ et $f_2=0$ par (\ref{eq:33}), on notera $f=\cD g$ ;
\item $\int_{M_0}g \,\vol^{g_0}=0$, $\int_{M_0}f \,\vol^{g_0}=0$ et $f$ satisfait la condition de Neumann ($f_1=0$), on notera $f=\cN g$.
\end{itemize}

Construisons alors un premier inverse \`a droite, $R_1$, pour
l'op\'erateur (\ref{eq:29}) : partant d'une matrice sym\'etrique \`a trace
nulle $v=(v_{ab}\omega_a\otimes\omega_b)$, satisfaisant (\ref{eq:30}), d\'efinissons
\begin{equation}
  \label{eq:34}
  R_1 v = d_+^*w, \text{ o\`u }w=
  \begin{pmatrix}
    \cD v_{11} & \cN v_{12} & \cN v_{13} \\
    \cN v_{21} & \cD v_{22} & \cD v_{23} \\
    \cN v_{31} & \cD v_{32} & \cD v_{33}
  \end{pmatrix}.
\end{equation}

Dans cette \'ecriture, la matrice est \'ecrite dans la base des $\omega_a$ et
repr\'esente donc un \'el\'ement de $\Sym_0^2\Omega_+$.

Les conditions (\ref{eq:30}) l\'egitiment l'emploi de la solution du
probl\`eme de Neumann sur les coefficients $v_{12}$ et $v_{13}$. Il est
possible de comprendre le choix du probl\`eme de Dirichlet ou de Neumann
par la parit\'e (\ref{eq:2}) attendue pour la solution.

L'op\'erateur $R_1$ n'est qu'une premi\`ere approximation \`a l'inverse \`a
construire. En effet, $\alpha=R_1v\notin \cT^{s+1,\ell}_\delta$ en g\'en\'eral : posons
$\alpha=(\alpha_a)=d_+^*w$, o\`u $w$ est d\'efini dans (\ref{eq:34}), et
$\alpha_a=\alpha_{a,j}e^j$, alors on calcule (voir (\ref{eq:39}))
\begin{equation}
  \label{eq:35}
  \alpha_{1,1}|_X = \partial_xw_{11} + X_3w_{12} -X_2w_{13} =: \varphi .
\end{equation}
Ce terme va \^etre corrig\'e en utilisant une libert\'e de choix sur $\alpha$,
provenant de l'action infinit\'esimale des diff\'eomorphismes : si $\xi$ est
un champ de vecteurs sur $M_0$, alors $(\imath_\xi\omega_a)$ correspond \`a modifier
les $\omega_a$ par l'action infinit\'esimale de $\xi$, donc $(\imath_\xi\omega_a)\in \ker
d_{g_0}P$.

% Section \ref{sec:analyse}, on montre que la restriction au bord est
% continue,
% $$ H^s(M_0)\to\FS^{s-\frac12}(X), $$
% et qu'existe un op\'erateur de prolongement
% $$ \FS^{s-\frac12}(X)\to H^s(M_0). $$

Par construction, $\varphi\in \FS^{\ell+2+\delta}$. Appliquons l'op\'erateur de
prolongement $E_2$ pour obtenir un prolongement $\tilde \varphi=E_2\varphi\in
\cH^{s+2,\ell}_{2,\delta}$ de $\varphi$ \`a l'int\'erieur de $M_0$, dans un
voisinage de $X$, et d\'efinissons le champ de vecteurs
\begin{equation}
  \label{eq:36}
  \xi = \tilde \varphi  x^{-1}\partial_x + (X_2 \tilde \varphi) X_2 + (X_3 \tilde \varphi) X_3 .
\end{equation}
Le premier coefficient du champ $\xi$ est calcul\'e de sorte que $\imath_\xi\omega_1=\varphi
x^{-1}\theta^1$ sur $X$. La pr\'esence de coefficients de $X_2$ et $X_3$
permet de pr\'eserver la structure de contact sur $X$, voir section
\ref{sec:les-deform-infin}. Alors l'inverse \`a droite voulu est donn\'e
par :
\begin{prop}\label{prop:inverse_a_droite}
  L'op\'erateur
  \begin{equation}
    \label{eq:37}
    Rv = R_1v - (\imath_\xi\omega_a)
  \end{equation}
  est un inverse \`a droite de $d_{g_0}P:\cT^{s+1,\ell}_\delta \to \cH^{s,\ell}_{\delta;0}(\Sym_0^2 \Omega_+)$.
\end{prop}

Le reste de cette section est consacr\'e \`a la d\'emonstration de la
proposition. On sait que $R$ est un inverse \`a droite, le probl\`eme
ici est de v\'erifier qu'il s'agit d'un op\'erateur entre les espaces
sp\'ecifi\'es.

Nous commen\c cons par calculer $\eta=d^*w$. Nous avons $\eta_b=\sum_1^3 I_a
dw_{ab}$. Pour obtenir le comportement pr\`es de $X$, observons que le
comportement asymptotique (\ref{eq:7}) implique que
$(e^0,e^1,e^2,e^3)$ est une base quaternionienne standard le long de
$X$, donc
\begin{equation}
  \label{eq:38}
  \begin{cases}
    J_ae^0=e^a + O(x^2), & a=1,2,3, \\
    J_ae^b=\epsilon_{abc}e^c + O(x^2), & (abc) \text{ permutation de }(123).
  \end{cases}
\end{equation}
Ici $O(x^2)$ vise les coefficients dans la base $(e^i)$ ; cette
d\'ecroissance provient de la propri\'et\'e de parit\'e (\ref{eq:2}) de la
m\'etrique.

Les coefficients de $\eta_b$ dans la base $(e^i)$ sont automatiquement dans l'espace $\cH^{s+1,\ell}_{2,\delta}$. Explicitons les deux premiers termes : la restriction \`a $X$ et la d\'eriv\'ee normale. Un calcul direct donne, modulo $x^2\cH^{s+1,\ell}_\delta$,
\begin{equation}\label{eq:39}
  \begin{split}
    \eta_b = & (-xX_1w_{1b}-X_2w_{2b}-X_3w_{3b}) dx + (\partial_xw_{1b}+X_3w_{2b}-X_2w_{3b}) x^{-1}\theta^1 \\
    & + (-X_3w_{1b}+\partial_xw_{2b}+xX_1w_{3b}) \theta^2 + (X_2w_{1b}-xX_1w_{2b}+\partial_xw_{3b}) \theta^3.
  \end{split}
\end{equation}

Analysons maintenant les $$\alpha_b=\eta_b-\iota_\xi\omega_b$$ pour v\'erifier les conditions du lemme \ref{lem:condition-1-formes}. R\'ecrivons, toujours modulo $x^2\cH^{s+1,\ell}_\delta$,
\begin{equation}\label{eq:40}
  \begin{split}
    \alpha_1 = & (-X_2w_{21}-X_3w_{31}) dx \\
    & + (\partial_xw_{11}+X_3w_{21}-X_2w_{31} - \tilde \varphi) x^{-1}\theta^1 \\
    & + (-X_3w_{11}+\partial_xw_{21}+xX_1w_{31}+xX_3\tilde \varphi) \theta^2 \\
    & + (X_2w_{11}-xX_1w_{21}+\partial_xw_{31}-xX_2\tilde \varphi) \theta^3.
  \end{split}
\end{equation}
Observons que la restriction \`a $X$ du coefficient de $\eta_1$ sur $e^1=x^{-1}\theta_1$ est exactement la fonction $\varphi$ d\'efinie par (\ref{eq:35}), donc le coefficient de $\alpha_1$ sur $e^1=x^{-1}\theta^1$ est nul, et en outre
$$ (\partial_x\alpha_{1,1})|_X=\partial_x^2w_{11}+\partial_xX_3w_{21}-\partial_xX_2w_{31}=0 $$
car $w_{21}$ et $w_{31}$ satisfont la condition de Neumann, et la condition de Dirichlet pour $w_{11}$ donne avec (\ref{eq:33}) l'annulation $\partial_x^2w_{11}|_X=0$. Les conditions au bord impliquent aussi que $\alpha_{1,2}|_X=\alpha_{1,3}|_X=0$, donc finalement $i^*\alpha_1=0$.

V\'erifions en outre pour $\alpha_1$ les annulations requises par le lemme \ref{lem:condition-1-formes} : en rempla\c cant $\varphi$ par sa valeur,
\begin{equation}
  \label{eq:41}
  (\partial_x\alpha_{1,2}-X_2\alpha_{1,0})|_X = (\partial_x^2+X_2^2+X_3^2)w_{21} + (X_1+X_2X_3-X_3X_2)w_{31}
\end{equation}
 qui est nulle, \`a cause de $X_1=-[X_2,X_3]$ et de la contrainte (\ref{eq:33}) sur $w_{21}$ ; la seconde annulation est similaire, et la derni\`ere est une cons\'equence de $i^*\alpha_1=0$ que nous avons d\'ej\`a vue.

Passons \`a $\alpha_2$ (le cas de $\alpha_3$ est similaire) : partons de la
formule, modulo des termes dans $x^2\cH^{s+1,\ell}_\delta$,
\begin{equation}
  \label{eq:42}
  \begin{split}
    \alpha_2 = &  (-xX_1w_{12}-X_2w_{22}-X_3w_{32}+xX_2\tilde \varphi) dx \\
    & + (\partial_xw_{12}+X_3w_{22}-X_2w_{32}-xX_3\tilde \varphi) x^{-1}\theta^1 \\
    & + (-X_3w_{12}+\partial_xw_{22}+xX_1w_{32}-\tilde \varphi) \theta^2 \\
    & + (X_2w_{12}-xX_1w_{22}+\partial_xw_{32}) \theta^3 .
  \end{split}
\end{equation}
Les conditions au bord donnent bien l'annulation sur $X$ du coefficient de $e^1=x^{-1}\theta^1$. Les deux premi\`eres annulations requises par le lemme \ref{lem:condition-1-formes} sont \'evidentes, et la derni\`ere r\'esulte d'un calcul direct. Cela conclut la preuve de la proposition \ref{prop:inverse_a_droite}.\qed

% Pour une fonction $f$, on a $d_+d_+^*(f\omega_1)=-d_+d^Cf=\frac12 (\Delta f)\omega_1$. La proposition \ref{prop:inverse_a_droite} et sa d\'emonstration ont alors la cons\'equence suivante.
% \begin{coro}
%   Soit une fonction $g\in \cH^{s,\ell}_\delta$. Alors il existe un unique couple $(f,\varphi)$, o\`u $f\in \cH^{s+2,\ell}_\delta$ et $\varphi\in \FS^{\ell+3+\delta}$ tel que 
% \end{coro}

\section{Les d\'eformations infinit\'esimales}
\label{sec:les-deform-infin}

Il r\'esulte du th\'eor\`eme \ref{th:submersion} que l'espace tangent \`a
$\cM^s$ est constitu\'e des diff\'erentielles ext\'erieures des triplets de
1-formes $(\eta_a)$, \`a coefficients $\cH^{s+1,\ell}_{2,\delta}$ dans la base $(e^i)$, tels
que $i^*\eta_1=0$, $d\eta_a\in x\cH^{s,\ell}_\delta$, et satisfaisant les \'equations
\begin{align}
  \omega^c \land d\eta^c & = 0 , \label{eq:43} \\
  \omega_1 \land d\eta^c + d\eta_1 \land \omega^c & = 0 , \label{eq:44} \\
  \omega_1 \land d\eta_1 + \frac12 \re (\omega^c \land d\overline{\eta^c}) & = 0. \label{eq:45}
\end{align}
Ici on a not\'e $\eta^c=\eta_2+i\eta_3$.

La premi\`ere \'equation, (\ref{eq:43}), dit juste que $d\eta^c$ est une
d\'eformation infinit\'esimale de la structure symplectique holomorphe
$\omega^c$. La condition au bord ($i^*d\eta^c$ verticale) dit que la structure
complexe continue \`a pr\'eserver la distribution de contact $H\subset TX$.

La deuxi\`eme \'equation, (\ref{eq:44}), est une condition de
compatibilit\'e de la forme de K\"ahler $\omega_1$ \`a la structure
complexe. Elle se r\'ecrit en termes de la partie de type (0,1) pour $J_1$ :
\begin{equation}
  \label{eq:46}
  \db(\eta_1^{0,1}\land \omega^c) = - \omega_1 \land d\eta^c.
\end{equation}

\begin{lemm}
  \'Etant donn\'e $\eta^c$, l'\'equation (\ref{eq:46}) a toujours des solutions
  $\eta_1$.
\end{lemm}
Ce lemme s'interpr\`ete en disant que, au moins au niveau infinit\'esimal, la classe $\zeta_1\in H^2(M_0,X)$ demeure de K\"ahler pour la d\'eformation infinit\'esimale donn\'ee par $\eta^c$. Cela est plausible car la forme $\omega^c$, de type (2,0), non nulle au bord, ne contribue pas \`a la cohomologie relative $H^2(M_0,X)$.
\begin{proof}
  Plut\^ot que de r\'esoudre (\ref{eq:46}), on se ram\`ene \`a un laplacien en
  consid\'erant l'\'equation
  \begin{equation}
    \label{eq:47}
    \db\db^*(f \vol^{g_0}) = \frac12 \omega_1 \land d\eta^c ,
  \end{equation}
  dont une solution $f$ produit une solution de (\ref{eq:46}) en
  posant
  \begin{equation}
    \label{eq:48}
    \eta^{0,1} = J_2 \partial f.
  \end{equation}
  R\'ecrivons l'\'equation (\ref{eq:47}) comme
  \begin{equation}
    \label{eq:49}
    \Delta f = -2 \Lambda d\eta^c.
  \end{equation}
  Or $\int_{M_0} \omega_1\land d\eta^c=0$, donc on peut prendre pour
  $f$ la solution du probl\`eme de Neumann. Comme les coefficients
  de $d\eta^c$ dans la base $(e^i\land e^j)$ sont dans $x\cH^{s,\ell}_\delta$, on a
  $\Lambda d\eta^c\in \cH^{s,\ell}_\delta$, et donc $f\in \cH^{s+2,\ell}_{3,\delta}$. Modulo des termes dans
  $x^2\cH^{s+1,\ell}_\delta$, on obtient 
  \begin{equation}
    \label{eq:50}
    J_2 \partial f = \big\{(\partial_x-ixX_1)f\big\} (\theta^2-i\theta^3) - \big\{(X_2-iX_3)f\big\} (dx-ix^{-1}\theta^1).
  \end{equation}
  Nous devons maintenant v\'erifier que $J_2 \partial f$ satisfait les
  conditions au bord voulues. Gr\^ace \`a la condition de Neumann, le
  premier coefficient $(\partial_x-ixX_1)f\in x\cH^{s+1,\ell}_\delta$. A priori, le second
  coefficient n'a pas de raison de s'annuler : ici on utilise le fait
  que l'\'equation \`a r\'esoudre est (\ref{eq:46}), donc on peut modifier
  $J_2 \partial f$ par un terme $\db g$. On pose alors
  $$ \eta^{0,1}=J_2 \partial f + \db (2xh), \quad h = E_2\big((X_2-iX_3)f|_X\big)\in \cH^{s+1,\ell}_{2,\delta}. $$
  Nous obtenons alors $\eta^{0,1}=a(dx-ix^{-1}\theta^1)+b(\theta^2-i\theta^3)$ modulo $x^2\cH^{s+1,\ell}_\delta$, avec
  \begin{equation}\label{eq:51}
  \begin{split}
    a &= -(X_2-iX_3)f + h + x(\partial_x+ixX_1) h \\
    b &= (\partial_x-ixX_1)f + x(X_2+iX_3)f.
  \end{split}
  \end{equation}
  \'Ecrivons $\eta=\re(a)dx+\im(a)x^{-1}\theta^1+\re(b)\theta^2+\im(b)\theta ^3$ ; la condition au bord $i^*\eta=0$ s'\'ecrit
  \begin{equation}
    \label{eq:52}
    \im(a)=\partial_x\im(a)=0, \quad b=0,
  \end{equation}
  tandis que les conditions du lemme \ref{lem:condition-1-formes} sont
  \begin{equation}
    \label{eq:53}
    \partial_x\re(b)-X_2\re(a)=0, \quad \partial_x\im(b)-X_3\re(a)=0.
  \end{equation}
  \`A partir de (\ref{eq:51}), les conditions (\ref{eq:52}) sont imm\'ediates, et on a m\^eme au bord $a=0$. La condition (\ref{eq:53}) se r\'eduit donc \`a $\partial_xb|_X=0$, et on calcule
  \begin{align*}
    \partial_xb &= \partial_x^2f-iX_1f+(X_2+iX_3)(X_2-iX_3)f \\
        &= (\partial_x^2+X_2^2+X_3^2)f-i(X_1+X_2X_3-X_3X_2)f.
  \end{align*}
  Comme $\Delta f\in \cH^{s,\ell}_\delta$, il faut que $(\partial_x^2+X_2^2+X_3^2)f$ s'annule sur $X$ ; le second terme s'annule aussi puisque $X_1=-[X_2,X_3]$.
\end{proof}

Supposons donn\'ee maintenant une solution $(\eta_1,\eta^c)$ de (\ref{eq:43}) et (\ref{eq:44}). Consid\'erons fix\'ee la variation de structure symplectique holomorphe, repr\'esent\'ee par $\eta^c$, et tentons de modifier $\eta_1$ de sorte de r\'esoudre aussi la derni\`ere \'equation (\ref{eq:45}), tout en pr\'eservant (\ref{eq:44}) : on a donc la flexibilit\'e de modifier $\eta_1^{0,1}$ par un terme $\db g$. Discutons la condition au bord sur $g$ : il faut $i^*\eta_1=0$ et donc $\db_Hg=0$ sur $X$, c'est-\`a-dire que $g|_X$ est une fonction holomorphe au sens CR sur $X$. 

Dans un premier temps, nous allons discuter uniquement les d\'eformations telles que $g|_X=0$, et nous \'etudierons les d\'eformations r\'esiduelles dans la section suivante.

 Si $g$ est r\'eelle, alors $\eta_1$ est modifi\'ee par $dg$ ce qui ne modifie pas $\omega_1$ ; en revanche, si $g=if$ est imaginaire pure, alors $\eta_1$ est modifi\'ee par $d^Cf$, et l'\'equation (\ref{eq:45}) devient
\begin{equation}
  \label{eq:54}
  \Lambda dd^Cf = -\frac12 \re(\omega^c\land d\overline{\eta^c}) - \Lambda d\eta_1.
\end{equation}
Ce n'est rien d'autre que la lin\'earisation de l'\'equation de Monge-Amp\`ere \`a r\'esoudre pour obtenir une m\'etrique k\"ahl\'erienne Ricci plate.

On peut r\'esoudre (\ref{eq:54}) avec la condition de Dirichlet $f|_X=0$. Cette solution satisfait $f\in x\cH^{s+2,\ell}_{2,\delta}$ et \`a nouveau, modulo des termes dans $x^2\cH^{s+1,\ell}_\delta$, la modification de $\eta_1$ est, pr\`es de $X$,
\begin{equation}
  \label{eq:55}
  d^Cf = - xX_1f dx + (\partial_xf) x^{-1}\theta^1 - X_3f \theta^2 + X_2f \theta^3.
\end{equation}
Tous les coefficients s'annulent sur $X$, sauf le terme singulier :
\begin{equation}
  \label{eq:56}
  (\partial_xf) x^{-1}\theta^1 .
\end{equation}
Comme on a vu, on peut encore modifier $\eta_1$ par un terme $dg$, mais cela ne permet pas de compenser ce terme singulier. Il y a donc une obstruction \`a r\'esoudre le probl\`eme : il n'existe pas pour tout $\eta^c$ de solution $\eta_1=d^Cf$ ($f|_X=0$) du probl\`eme (\ref{eq:44})--(\ref{eq:45}).

N\'eanmoins, on peut \'eliminer le terme (\ref{eq:56}) par un diff\'eomorphisme infinit\'esimal agissant sur le triplet $(\omega_a)$ : posant $\psi=E_2(\partial_xf|_X)\in \cH^{s+2,\ell}_{2,\delta}$, on d\'efinit le champ de vecteurs
\begin{equation}
  \label{eq:57}
  \xi = \psi x^{-1}\partial_x + (X_2 \psi)X_2 + (X_3 \psi)X_3 .
\end{equation}
\begin{lemm}\label{lem:1-formes}
  Les 1-formes $(\eta^c-\imath_\xi\omega^c,\eta_1+d^Cf-\imath_\xi\omega_1)$ sont solutions du
  syst\`eme (\ref{eq:43})--(\ref{eq:45}) et satisfont les conditions au bord.
\end{lemm}
\begin{proof}
  Il reste juste \`a v\'erifier les conditions au bord. Modulo des termes $O(x^2)$, on a
  \begin{align*}
    d^Cf-\iota_\xi\omega_1 = & -xX_1f dx + (\partial_xf-\psi)x^{-1}\theta^1 \\ &- X_3(f-x\psi)\theta^2 + X_2(f-x\psi)\theta^3, \\
     \imath_\xi\omega^c = & - x(X_2+iX_3)\psi (dx + i x^{-1}\theta^1) + \psi (\theta^2+i\theta^3) ,
  \end{align*}
  et les conditions du lemme \ref{lem:condition-1-formes} se v\'erifient facilement.
\end{proof}
Bien entendu, l'annulation qui vient d'\^etre montr\'ee est la raison de la pr\'esence des termes $(X_b\psi) X_b$ dans la formule (\ref{eq:57}), comme dans (\ref{eq:36}).

Synth\'etisons ce que nous venons de d\'emontrer. \'Etant donn\'ee une d\'eformation
holomorphe symplectique infinit\'esimale $d\eta^c$, nous pouvons compl\'eter $\eta^c$ en une solution du syst\`eme (\ref{eq:43})--(\ref{eq:45}), \`a la condition d'autoriser
l'action de diff\'eomorphismes infinit\'esimaux comme dans
(\ref{eq:57}). Or, le changement de variable $y=x^2/2$ ($x>0$) fait
disparaitre la singularit\'e de la structure holomorphe symplectique sur
$X$, et $x^{-1}\partial_x=\partial_y$, donc nous voyons que cela correspond \`a
d\'eplacer infinit\'esimalement le bord de $M_s$ ; en outre, la fonction
$\psi$ dans (\ref{eq:57}) \'etant parfaitement d\'etermin\'ee, le d\'eplacement
infinit\'esimal est uniquement d\'etermin\'e. Cela ach\`eve la d\'emonstration
de la premi\`ere partie du th\'eor\`eme \ref{theo:principal}, et justifie la
question pos\'ee dans l'introduction, qui est une question de r\'esolution
du probl\`eme de Monge-Amp\`ere dans cette situation, avec fronti\`ere
libre.

Concluons en remarquant que cette question se comprend tr\`es bien dans
le formalisme de Donaldson, consistant \`a chercher une m\'etrique
K\"ahler-Einstein dans l'orbite complexifi\'ee du groupe des
symplectomorphismes. Dans notre situation, il faut penser au groupe
des symplectomorphismes comme induisant au bord un
contactomorphisme. Or la complexification d'un contactomorphisme
infinit\'esimal d\'eplace n\'ecessairement le bord : en fait, le champ $\xi$
dans (\ref{eq:57}) n'est autre que le complexifi\'e du contactomorphisme
infinit\'esimal
$$ - \psi X_1 + (X_3\psi)X_2 - (X_2\psi)X_3 . $$
Comme on l'a vu, les contactomorphismes complexifi\'es continuent \`a
induire la structure de contact fix\'ee $H$ sur $X$. Voyant ainsi $X\subset N$
comme hypersurface dans une vari\'et\'e holomorphe symplectique (sans
singularit\'e), il est naturel d'identifier la complexification du
groupe des contactomorphismes aux plongements $\phi:X\to N$ tels que la
structure CR induite par $\phi$ sur $X$ demeure d\'efinie sur la structure
de contact $H$. Les champs de vecteurs (\ref{eq:57}) en sont
exactement la version infinit\'esimale.

\section{Les d\'eformations infinit\'esimales de Hitchin}
\label{sec:lespace-des-modules-1}

Examinons \`a pr\'esent les solutions additionnelles du syst\`eme (\ref{eq:43})-~(\ref{eq:45}) provenant d'un potentiel non nul sur le bord $X$. Comme on a vu, la modification de $\eta^{0,1}$ par $\db g$ ne modifie pas les \'equations (\ref{eq:43}) et (\ref{eq:44}), et la condition au bord $i^*\eta_1=0$ exige $\db_H(g|_X)=0$.

Nous restreignons la discussion au cas mod\`ele, sur le fibr\'e en disques de $T^*\Sigma$. Si $\omega_\Sigma$ est la forme de K\"ahler \`a courbure $-1$ sur $\Sigma$, alors il r\'esulte des formules dans \cite{BiqGau97} qu'on peut \'ecrire la m\'etrique hyperk\"ahl\'erienne pli\'ee par la formule
\begin{equation}
  \label{eq:58}
  \omega_1 = \sqrt{1-r^2} p^*\omega_\Sigma + \frac1{\sqrt{1-r^2}} r dr \land \eta,
\end{equation}
o\`u $p:T^*\Sigma\to\Sigma$ est la projection, $r$ est la distance dans la fibre de $T^*\Sigma$, et $\eta$ est la 1-forme de connexion sur le fibr\'e en cercles, donc $d\eta=-\omega_\Sigma$. (On a $\eta=-\theta^1$, o\`u $\theta^1$ est la forme d\'efinie section \ref{sec:la-geometrie-au}).

Soit $g\in \FS^{\ell+3+\delta}$ une fonction CR-holomorphe sur $X$, alors on peut r\'esoudre le probl\`eme de Dirichlet
\begin{equation}
 \Delta \tilde g=0, \quad \tilde g|_X=g,\label{eq:59}
\end{equation}
avec $\tilde g\in \cH^{s+2,\ell}_{3,\delta}$. Posant $\eta_1=\re (\db \tilde g)$, on obtient
$$ \omega_1\land d\eta_1 = \re (\omega_1\land \partial\db \tilde g) = 0, $$
donc on obtient une solution infinit\'esimale $(\eta^c=0,\eta_1=\re(\db \tilde g))$ du syst\`eme lin\'earis\'e (\ref{eq:43})-(\ref{eq:45}). Ces solutions ne modifient pas la structure holomorphe symplectique, au moins infinit\'esimalement.

De mani\`ere explicite, les fonctions CR-holomorphes sur $X$ ont une
d\'ecomposition en s\'eries de Fourier, $g=\sum_{n\leq 0}g_n$ et $g_n\in
H^0(\Sigma,K^{-n})$. Pour $n=0$, la fonction $g_0$ est constante, donc son
extension $\tilde g_0$ aussi, donc $\eta_1=0$. Supposons donc $n<0$.
 Sans rentrer dans le d\'etail des calculs, on voit que $\Delta$ pr\'eserve les fonctions du type $a(r)g_n$, donc l'extension $\tilde g_n$ de $g_n$ est du m\^eme type. Il en r\'esulte aussi que $\db\tilde g_n$ est encore du m\^eme type :
\begin{equation}
  \label{eq:60}
  \db \tilde g_n = \varphi_n(r) g_n \big(\frac{dr}r - i \eta\big).
\end{equation}
\`A partir de l\`a, il est facile d'expliciter $\tilde g$ mais nous n'avons pas besoin de la formule pr\'ecise. L'\'equation $\Lambda\partial\db \tilde g_n=0$ m\`ene rapidement \`a $\varphi_n(r)=\Phi_n\frac{r^{-n}}{\sqrt{1-r^2}}$ pour une certaine constante $\Phi_n$. Finalement, on obtient
\begin{equation}
 \db \tilde g = -\iota_\xi\omega_1\label{eq:61}
\end{equation}
avec
\begin{equation}
  \label{eq:62}
    \xi=\sum_{n\leq 0}\xi_n, \quad \xi_n = 2i \Phi_n g_n \frac{r^{-n}}{\bar w}\frac \partial{\partial w},
\end{equation}
o\`u $w\partial_w$ est le champ de vecteurs (holomorphe) d'homoth\'etie dans les
fibres. La fonction $r^{-n}g_n/\bar w$ co\"\i ncide, \`a une constante pr\`es,
avec $\bar w^{-n-1}$ sur chaque disque.

Pour v\'erifier nos conditions au bord, transformons la solution
pr\'ec\'edente par l'action infinit\'esimale de $\xi$ pour obtenir plut\^ot la
solution $(\eta^c=\iota_\xi\omega^c,\eta_1=0)$. Remarquant que pour $n=-1$, le champ
$\xi_1$ pr\'eserve $\omega^c$, on obtient :
\begin{prop}\label{prop:Hitchin-infinitesimal}
  Pour toute fonction $g\in \FS^{\ell+3+\delta}$ CR-holomorphe sur $X$,
  d\'efinissons $\xi$ par (\ref{eq:61}), o\`u $\tilde g$ est l'extension
  harmonique de $g$, alors $(\eta^c=\iota_\xi\omega^c,\eta_1=0)$ d\'efinit un vecteur
  tangent \`a $\cM_\delta^{s,\ell}$, c'est-\`a-dire est une solution
  infinit\'esimale des \'equations et satisfait les conditions au bord. La
  variation est non triviale pour les fr\'equences diff\'erentes de $0$ et
  $-1$, donc on obtient des d\'eformations infinit\'esimales param\'etr\'ees par
  \begin{equation}\label{eq:63}
    \oplus_{n\leq -2} H^0(\Sigma,K^{-n}).
  \end{equation}\qed
\end{prop}

Ces d\'eformations infinit\'esimales ont \'et\'e trouv\'ees par Hitchin
\cite[\S9]{Hit14}. Notre approche dans cette section montre qu'on
obtient ainsi \emph{toutes} les d\'eformations infinit\'esimales qui
restent sur la vari\'et\'e holomorphe symplectique $T^*\Sigma$. Gr\^ace \`a la
proposition ci-dessus, elles donnent des vecteurs tangents dans
$\cM_\delta^{s,\ell}$, et donc sont tangentes \`a des d\'eformations par de vraies
m\'etriques hyperk\"ahl\'eriennes pli\'ees. Les sections suivantes ont pour objet de montrer que ces m\'etriques hyperk\"ahl\'eriennes peuvent \^etre prises de sorte que leur
structure holomorphe symplectique soit bien celle d'un domaine de $T^*\Sigma$.

\begin{rema}\label{rem:pol-inv}
  On peut v\'erifier que les polyn\^omes invariants d\'efinis par Hitchin \cite[\S8.2]{Hit14}, \`a savoir
  \begin{equation}
 p_n = \big( \int_{M_0/\Sigma} w^n \omega_1 \big) dz^n \in H^0(\Sigma,K^n),\label{eq:105}
\end{equation}
\'evalu\'es sur ces d\'eformations infinit\'esimales, redonnent bien le param\`etre $g\in \oplus_{n\geq 2}H^0(\Sigma,K^n)$. En effet, faisons le calcul en consid\'erant que $g$ param\`etre la d\'eformation infinit\'esimale $(\eta^c=0,\eta_1=-\imath_\xi\omega_1)$, que nous choisissons ainsi puisque la structure de cotangent de $T^*\Sigma$ est pr\'eserv\'ee. Dans le calcul de la variation infinit\'esimale de $p_m$, il faut faire attention \`a prendre en compte la variation du domaine dans $T^*\Sigma$, qui se traduit par une contribution $\int_{X/\Sigma}w^n\imath_\xi\omega_1$ dans l'int\'egrale. Ainsi, la variation infinit\'esimale de $p_m$ est
\begin{align*}
  \dot p_n &= \big( \int_{M_0/\Sigma} w^nd\eta_1 - \int_{X/\Sigma} w^n \eta_1 \big) dz^n \\
           &= \big( \int_{M_0/\Sigma} -w^n \frac{dw}w \land \eta_1 \big) dz^n \\
           &= \big( \int_{M_0/\Sigma} -w^n \frac{dw}w \land \db\tilde g \big) dz^n \\
           &= \big( \int_{X/\Sigma} w^n g \frac{dw}w \big) dz^n \\
           &= g_n.
\end{align*}
\end{rema}

\section{Param\'etrisation des domaines dans le cotangent}
\label{sec:param-des-doma}

Nous proposons ici une param\'etrisation des d\'eformations du domaine
$M_s\subset T^*\Sigma$. Le r\'esultat essentiel  est le th\'eor\`eme
\ref{th:domaines} qui permet une param\'etrisation sans perte de
d\'eriv\'ees. Pour l'\'etude g\'en\'erale de toutes les d\'eformations d'un tel
domaine pseudoconcave, on pourra consulter \cite{EpsHen00} .

Une approche plus simple consisterait \`a param\'etrer ces d\'eformations
par les plongements $\varphi:X\to N$ qui induisent la m\^eme structure de contact
sur $X$. Si l'image est restreinte \`a $X$, on param\`etre ainsi les
contactomorphismes par une fonction r\'eelle (voir \cite{BlaDuc14}, ou
\cite[\S5]{Biq02} pour le cas $S^1$ invariant) : cette param\'etrisation
peut se faire sans perte de d\'eriv\'ees, et est bas\'ee sur les propri\'et\'es
du complexe de Rumin. Comme expliqu\'e \`a la fin de la section
\ref{sec:les-deform-infin}, le cas g\'en\'eral correspond \`a la
complexification du groupe des contactomorphismes, le complexe de
Rumin est \`a remplacer par le complexe du $\db_H$ sur la vari\'et\'e CR $X$,
mais les mauvaises propri\'et\'es analytiques de $\db_H$ semblent emp\^echer
une construction similaire. Nous adoptons donc dans cette section une
approche compl\`etement diff\'erente.

Une d\'eformation $J$ de la structure complexe $J_0$ de $M_s$ est
param\'etr\'ee par un tenseur $\phi\in \Omega^{0,1}\otimes T^{1,0}$ (les bi-degr\'es sont
pour $J_0$) tel que l'espace des vecteurs de type $(0,1)$ pour $J$
soit le graphe de $\phi:T^{0,1}\to T^{1,0}$,
$$  T^{0,1}_J = \{ \xi + \phi_\xi, \xi\in T^{0,1}_{J_0} \}. $$
A priori, un tel $\phi$ ne d\'efinit qu'une structure presque complexe $J$
; elle est int\'egrable si
\begin{equation}
 \db \phi + \frac12 [\phi,\phi] = 0,\label{eq:64}
\end{equation}
o\`u $[\phi,\phi]$ fait intervenir le produit ext\'erieur des formes et
le crochet des champs de vecteurs.

La structure complexe $J$ induit sur le bord $X=\partial M_s$ une structure
CR. Quitte \`a agir par un diff\'eomorphisme, on peut supposer que la
structure CR garde la m\^eme distribution de contact sous-jacente
$H$. En outre, comme $M_s$ est un fibr\'e holomorphe en disques, nous
avons une d\'ecomposition globale, $S^1$-invariante,
$$ TM = H \oplus V, $$
o\`u $V=\ker p_*$ et $H$ est l'horizontal fourni par la structure de
contact.

Par \cite[Th\'eor\`eme 4.1]{Biq02}, toute petite d\'eformation de $J$ pour
laquelle $\Sigma\subset M_s$ demeure une sous-vari\'et\'e holomorphe, apr\`es action
d'un diff\'eomorphisme unique modulo $S^1$, s'\'ecrit dans cette
d\'ecomposition sous la forme
\begin{equation}
 \phi=
\begin{pmatrix}
  \psi&0\\0&0
\end{pmatrix},\label{eq:65}
\end{equation}
o\`u $\psi\in \Omega^{0,1}H\otimes H^{1,0}=p^*(\Omega^{0,1}_\Sigma\otimes T^{1,0}_\Sigma)$ est holomorphe le
long de chaque disque de la fibration $p$ (cela a un sens puisque le
fibr\'e $\Omega^{0,1}_\Sigma\otimes T^{1,0}_\Sigma$ est trivial le long de chaque disque) :
donc, d\'ecomposant en s\'eries de Fourier pour l'action de $S^1$,
\begin{equation}
  \label{eq:66}
  \psi=\sum_{n\geq 0} \psi_n, \text{ o\`u } \psi_n|_{p^{-1}(x)} = F_n(x) w^n,
\end{equation}
o\`u $w$ est un choix de coordonn\'ee holomorphe sur le disque
$p^{-1}(x)$. Plus intrins\`equement, on peut voir $w$ comme un point de
l'espace total du fibr\'e $K$, et $F_n$ comme une section sur $\Sigma$
de $K^{-n}\otimes\Omega^{0,1}_\Sigma\otimes T^{1,0}_\Sigma=K^{-n-1}\otimes\Omega^{0,1}_\Sigma$. 

R\'eciproquement, la donn\'ee d'un tel $\psi$, holomorphe le long des
disques de la fibration, induit une d\'eformation complexe $\phi$ donn\'ee
par (\ref{eq:65}), satisfaisant $[\phi,\phi]=0$ et donc l'\'equation
d'int\'egrabilit\'e (\ref{eq:64}).

Parmi ces d\'eformations holomorphes, cherchons celles qui
demeurent un domaine dans un cotangent holomorphe $T^*\Sigma$. Nous demandons ainsi l'existence d'une projection $\pi:M_s\to\Sigma$ et d'une 2-forme complexe $\Omega$, telles que
\begin{equation}\label{eq:67}
  \begin{split}
    \Omega &\in \Omega^{2,0}_J ; \\
    d\Omega &= 0 ; \\
    \db_J\pi &=0.
  \end{split}
\end{equation}
On obtient alors imm\'ediatement que $\Omega=d\Theta$ avec $\Theta$ une 1-forme s'annulant sur les fibres de $\pi$ et sur $\Sigma$ : la forme $\Theta$ identifie alors $M_s$ avec $T^*\Sigma$, dont elle appara\^\i t comme la forme de Liouville.

Analysons tout d'abord les conditions sur $\Omega$.
La premi\`ere condition permet d'\'ecrire, pour une (1,0)-forme horizontale $\alpha$,
\begin{align}
  \label{eq:68}
  \Omega &= (\alpha - \psi \lrcorner \alpha)\land \eta^{1,0}, \\
\intertext{d'o\`u}
 d\Omega &= d(\alpha - \psi \lrcorner \alpha) \land \eta^{1,0} .\label{eq:69}
\end{align}
La deuxi\`eme condition, \'equivalente \`a $\db_J\Omega=0$, se traduit par
\begin{equation}\label{eq:70}
  \begin{split}
    \imath_{\bar w \partial_{\bar w}}d\Omega&=0, \\
    \imath_{\xi+\psi_\xi}d\Omega&=0 \text{ pour }\xi\in H^{1,0}.
  \end{split}
\end{equation}
La premi\`ere \'equation dans (\ref{eq:70}) m\`ene \`a
\begin{equation}
  \label{eq:71}
  \partial_{\bar w}(\alpha-\psi\lrcorner \alpha)=0,
\end{equation}
c'est-\`a-dire $\alpha$ et $\psi\lrcorner \alpha$ sont holomorphes le long des disques de la
fibration, ce qui, puisque $\psi$ est d\'ej\`a holomorphe le long des disques, est \'equivalent \`a $\alpha$ holomorphe le long des disques ; on a
\begin{equation}
  \label{eq:72}
  \imath_{\xi+\psi_\xi}d\Omega = \big( \imath_{\xi+\psi_\xi}d(\alpha-\psi\lrcorner \alpha) \big) \land \eta^{1,0} ;
\end{equation}
comme $\imath_{\xi+\psi_\xi}d\Omega\in \Omega^{2,0}_J$, son annulation est \'equivalente \`a celle
de sa projection sur $\Omega^{2,0}_{J_0}$, donc, notant $d_H$, $\db_H$ et $\partial_H$ les
restrictions \`a $H$ de $d$, $\db$ et $\partial$, la seconde \'equation de (\ref{eq:70}) est \'equivalente \`a
\begin{equation}\label{eq:73} 
    \db_H \alpha - \partial_H(\psi\lrcorner \alpha) = 0.
\end{equation}

Nous pouvons r\'esumer ces observations dans le lemme suivant.
\begin{lemm}
  Les 2-formes sur $M_s$ satisfaisant les deux premi\`eres \'equations de (\ref{eq:67}) sont en correspondance avec les (1,0)-formes horizontales $\alpha$ sur $X$, telles que
  \begin{enumerate}
  \item $\alpha$ n'a des coefficients non nuls que pour des fr\'equences strictement positives ;
  \item $\alpha$ satisfait l'\'equation (\ref{eq:73}) sur $X$, qui n'est
    autre que l'\'equation $\db_{H,J}(\alpha-\psi \lrcorner \alpha)=0$.
  \end{enumerate}
\end{lemm}
\begin{proof}
  Puisque $\alpha$ est holomorphe le long des disques, elle n'a de coefficients de Fourier non nuls qu'en fr\'equences positives. Sachant que $\eta^{1,0}$ s'identifie \`a $\frac{dw}{2iw}$ sur chaque disque, pour que $\Omega$ s'\'etende \`a la section nulle, il faut que les coefficients invariants par rotation s'annulent aussi, donc $\alpha$ n'a de coefficients non nuls qu'en fr\'equences strictement positives.
Dans ces conditions, puisque $\psi$ aussi est holomorphe le long des disques, le syst\`eme (\ref{eq:73}) est satisfait sur $M_s$ si et seulement s'il est satisfait sur le bord $X$.
\end{proof}

Revenons maintenant \`a l'\'equation sur $\pi$, un peu plus subtile \`a analyser. Commen\c cons par regarder les applications $\pi:M_s\to\Sigma$, d\'eformations de $p$, et holomorphes le long des disques. \'Etant donn\'e un point $\sigma\in \Sigma$, on peut choisir une coordonn\'ee locale $z$ autour de $\sigma$, et, au dessus d'un voisinage de $\sigma$, les projections $\pi$, holomorphes verticalement, s'identifient aux fonctions $z\circ \pi$ \`a valeurs dans $\setC$, holomorphes le long de chaque disque. Il appara\^\i t ainsi que l'espace des applications $\pi:M\to\Sigma$, holomorphes verticalement, de r\'egularit\'e $\FS^m$ sur $X$, est une vari\'et\'e banachique bien d\'efinie que nous noterons $\cP^m$, et dont l'espace tangent en $p$ s'identifie aux sections sur $X$ de $p^*T\Sigma$, de r\'egularit\'e $\FS^m$, dont les coefficients non nuls sont en fr\'equences positives---nous noterons cet espace $\FS^m_{\geq 0}(p^*T\Sigma)=\FS^m_{\geq 0}(H^{1,0})$.

Les solutions de la troisi\`eme \'equation de (\ref{eq:67}) s'identifient \`a pr\'esent aux applications $\pi\in \cP^m$, telles que sur le bord $X$ on ait pour tout $\xi\in H^{0,1}$,
\begin{equation}
  \label{eq:74}
  \pi_*(\xi+\psi_\xi)=0,
\end{equation}
ce qui, \`a nouveau, n'est autre que l'\'equation $\db_{H,J}\pi=0$ sur $X$.
En effet, en choisissant localement une coordonn\'ee holomorphe $z$ sur $\Sigma$, on voit qu'un \'el\'ement de $\cP^m$ satisfait l'\'equation (\ref{eq:74}) si et seulement s'il la satisfait sur $X$.

Il y a des solutions \'evidentes au syst\`eme (\ref{eq:67}), provenant de
la d\'eformation de l'hypersurface $X$ dans $T^*\Sigma$. Bien entendu, une
telle d\'eformation doit \^etre suivie d'un diff\'eomorphisme qui ram\`ene le
domaine d\'elimit\'e \`a la jauge particuli\`ere satisfaisant (\ref{eq:65}).

Explicitons ces solutions. Un contactomorphisme infinit\'esimal sur $X$
est de la forme $\xi=gR-\sharp d_Hg$, o\`u $g$ est une fonction r\'eelle sur $X$,
et $\sharp:\Omega^1H\to H$ est d\'efini par $\imath_{\sharp\alpha}d\eta=\alpha$. Il agit sur l'espace des
structures complexes infinit\'esimales sur $X$ par
\begin{equation}
\dot \psi=\db_H\sharp\db_Hg.\label{eq:75}
\end{equation}
Cette action se complexifie en d\'ecidant que la fonction $g$ peut \^etre
\`a valeurs complexes : l'action infinit\'esimale est alors celle de la
partie r\'eelle du vecteur de type $(1,0)$ donn\'e par
\begin{equation}
\xi^{1,0}=2(gR^{1,0}-\sharp\db_Hg)=2(igw\partial_w-\sharp\db_Hg),\label{eq:76}
\end{equation}
et l'action sur les structures CR reste donn\'ee par la formule
(\ref{eq:75}).

On a d\'ej\`a vu que ces champs de vecteurs le long de $X$
sont la version infinit\'esimale des plongements $\varphi:X\to T^*\Sigma$
telles que la structure CR induite par $\varphi$ sur $X$ garde $H$ comme
structure de contact sous-jacente, qui sont un analogue de la complexification du groupe des
contactomorphismes ; l'action (\ref{eq:75}) est exactement la complexification de l'action infinit\'esimale des contactomorphismes.

Comme il se doit, l'action des contactomorphismes ne pr\'eserve pas la
jauge (les coefficients de Fourier positifs). En revanche, l'action
complexifi\'ee infinit\'esimale des fonctions $g$ \`a fr\'equences positives
pr\'eserve cette jauge, et la proposition suivante montre qu'on obtient
ainsi toutes les d\'eformations. Notant $\FS^m_{\geq 0}$ (resp. $\FS^m_{>0}$) les espaces de sections dont les coefficients non nuls se trouvent uniquement en fr\'equences positives (resp. strictement positives) :
\begin{prop}\label{prop:param-domaines}
  Soit $m \gg 0$. L'espace des $$(\psi,\alpha,\pi)\in \FS^m_{\geq 0}\times \FS^m_{> 0}\times
  \cP^{m+1}$$ satisfaisant (\ref{eq:73}) et (\ref{eq:74}) est une vari\'et\'e
  d'espace tangent param\'etr\'e par l'action infinit\'esimale des vecteurs
  donn\'es par la formule (\ref{eq:76}), o\`u $g\in \FS^{m+2}_{\geq 0}$. Les formules sont
  $$ \dot \psi = \db_H \sharp \db_H g, \quad
     \dot \alpha = iR\cdot (g\Theta^0) - \partial_H\Lambda_\Sigma(\db_Hg\land \Theta^0), \quad
     \dot \pi = - \sharp \db_H g. $$
\end{prop}
Ici $\Theta^0$ d\'esigne la forme de Liouville initiale de $T^*\Sigma$.
On remarquera aussi que la fonction constante $g\in \setR$ correspond \`a
multiplier $\Theta^0$ par une constante imaginaire pure, ce qui ne change pas le domaine holomorphe symplectique. Cela correspond \`a l'ambigu\"\i t\'e de jauge d\^ue \`a l'action de $S^1$, donc on peut imposer $\re \int_Xg=0$ quand on param\`etre les d\'eformations holomorphes symplectiques de $M_s$.
\begin{proof}
  Nous consid\'erons donc l'op\'erateur
  \begin{multline}
Q:\FS^m_{\geq 0}(\Omega^{0,1}H\otimes\Omega^{1,0}H)\times \FS^m_{> 0}(\Omega^{1,0}H)\times
  \cP^{m+1}\\\longrightarrow\FS^{m-1}_{> 0}(\Omega^{0,1}H\otimes\Omega^{1,0}H)\times \FS^m_{\geq 0}(\Omega^{0,1}\otimes H^{1,0}),\label{eq:77}
\end{multline}
  d\'efini par
  \begin{equation}
    \label{eq:78}
    Q(\psi,\alpha,\pi) = (\db_{H,J}(\alpha-\psi\lrcorner \alpha),\db_{H,J}\pi),
  \end{equation}
de lin\'earisation
$$ L(\dot \psi,\dot \alpha,\dot \pi) =
(\db_H \dot \alpha- \partial_H(\dot \psi\lrcorner \Theta^0),\db_H\dot \pi + \dot \psi). $$
La d\'emonstration consiste maintenant \`a montrer que $L$ est surjective et \`a identifier $\ker L$.

L'op\'erateur $\db_H$ sur $X$ a une image ferm\'ee, mais de codimension infinie. Tout est explicite dans la d\'ecomposition en s\'eries de Fourier : sur un fibr\'e holomorphe $L$ provenant de $\Sigma$, le conoyau en fr\'equence $n$ est $H^1(K^{-n}\otimes L)=H^0(K^{n+1}\otimes L^*)^*$. Pour l'op\'erateur $\partial_H$, de mani\`ere similaire, l'image est ferm\'ee mais le conoyau en fr\'equence $n$ est $H^0(K^{-n-1}\otimes L)^*$.

En particulier, passant \`a la conjugaison, nous voyons que le conoyau de $\partial_H:\Omega^{0,1}H\to\Omega^{1,1}H$ est $H^0(K^{-n})^*$, et en particulier s'annule en fr\'equences $n>0$, ce qui donne la surjectivit\'e sur le premier facteur. En revanche, le noyau de $\partial_H$ en fr\'equence $n$ s'identifie \`a $H^0(K^{n+1})^*$.

Pour le second facteur, observons que le conoyau de $\dot \pi\mapsto\db_H\dot \pi$ s'identifie \`a $H^0(K^{n+1})^*$ en fr\'equence $n$, exactement compens\'e, comme on vient de le voir, par le noyau de $\dot \psi\mapsto\partial_H(\dot \psi\lrcorner \Theta^0)$. 

Par cons\'equent, $L$ est surjective, et l'espace des solutions du syst\`eme (\ref{eq:73}) est
une sous-vari\'et\'e, d'espace tangent \'egal \`a $\ker L$, que nous
d\'eterminons \`a pr\'esent. Soit donc $(\dot \psi,\dot \alpha,\dot \pi)\in \ker L$,
il faut donc que $\dot \psi=-\db_H\dot \pi$, puis
\begin{equation}
  \label{eq:79}
  \db_H \dot \alpha = - \partial_H \db_H f, \quad f=\dot \pi\lrcorner \Theta^0. 
\end{equation}
Utilisant la formule $\partial_H\db_Hf+\db_H\partial_Hf=d_H^2f=R\cdot f \omega_\Sigma$ pour toute fonction $f$, on
obtient
\begin{equation}
  \label{eq:80}
  \db_H(\alpha-\partial_Hf) = -R\cdot f \omega_\Sigma.
\end{equation}
Rappelons \`a nouveau que l'image de $\db_H$ est ferm\'ee, donc on peut
d\'ecomposer en somme orthogonale $\FS^{m+1}(\Omega^{1,1}H) = \im \db_H \oplus \ker \db_H^*$. Comme la d\'erivation par $R$ commute avec $\db_H$ et $\db_H^*$, elle envoie cette d\'ecomposition de $\FS^{m+1}$ sur la m\^eme d\'ecomposition de $\FS^{m-1}$. Par cons\'equent, l'\'egalit\'e (\ref{eq:80}) impose $f \omega_\Sigma\in \im\db_H$, donc il existe $g\in \FS^{m+2}_{> 0}$, unique \`a constante additive pr\`es, telle que
$ f \omega_\Sigma = -\db_H (g \Theta^0)$, ce qui s'\'ecrit encore, si $\Theta_0\in H^{1,0}$ est dual \`a $\Theta^0$,
$$ f\Theta_0 = - \sharp \db_H g. $$
\`A partir de l\`a, on r\'ecup\`ere $\dot \pi=-\sharp \db_Hg$, puis $\dot \psi=\db_H\sharp \db_Hg$, enfin, \`a partir de (\ref{eq:80}), on calcule
$$ \dot \alpha = iR\cdot (g\Theta^0) - \partial_H\Lambda_\Sigma(\db_Hg\land \Theta^0) ; $$
on peut v\'erifier que $(\dot \alpha-\dot \psi\lrcorner \Theta^0)\land \eta^{1,0}$ n'est
autre que l'action infinit\'esimale de $\xi$ sur la forme symplectique initiale $\Omega^0=d\Theta^0$ de $T^*\Sigma$.
\end{proof}

Nous synth\'etisons les r\'esultats de cette section de la mani\`ere suivante. Notons $\cQ^{s,\ell}_{\delta, \setC}$ l'espace des 2-formes complexes $\omega^c=\omega_2+i\omega_3$ telles que $(\omega_1=0,\omega_2,\omega_3)\in \cQ^{s,\ell}_\delta$, et $\cM^{s,\ell}_{\delta, \setC}\subset \cQ^{s,\ell}_{\delta, \setC}$ l'espace des 2-formes $\omega^c$ telles que $(\omega^c)^2=0$. Enfin, notons $\FS^m_{\geq 0}(X)_0$ l'espace des fonctions $f\in \FS^m_{\geq 0}$ telles que $\re \int_Xg=0$.
\begin{theo}\label{th:domaines}
  On pose $\delta=\frac12$ et $s=\ell+\frac12$.
  Il existe une sous-vari\'et\'e $\cC^{s,\ell}_\delta\subset \cM^{s,\ell}_{\delta, \setC}$ d\'efinie pr\`es du mod\`ele $M_s$, qui contient toutes les petites d\'eformations de $M_s$ provenant d'une variation de $X$ dans $T^*\Sigma$, et est transverse \`a l'action des diff\'eomorphismes.

En outre, il existe une param\'etrisation $\mho:\FS^{\ell+2+\delta}_{\geq
  0}(X)_0\to\cC^{s,\ell}_{\delta, \setC}$, d\'efinie sur un voisinage de $0$, telle
que $d\mho(g)$ est la d\'eformation infinit\'esimale de la forme initiale $\Omega^0$ par l'action
du champ de vecteurs
$$ \xi=2\re(igw\partial_w-\sharp \db_Hg),$$
prolong\'e holomorphiquement disque \`a disque le long des fibres de la projection $p:M_s\to\Sigma$.
\end{theo}
\begin{proof}
  \`A partir de la proposition \ref{prop:param-domaines}, nous obtenons
  des d\'eformations de $\omega^c$ param\'etr\'ees par une fonction $g\in
  \FS^{\ell+2+\delta}_{\geq 0}(X)_0$. Malheureusement l'extension holomorphe
  disque \`a disque donne une r\'egularit\'e \`a l'int\'erieur qui est la m\^eme
  que celle au bord, donc $\omega^c\in \cC^{\ell+\delta,\ell}_{\delta, \setC}$ seulement. 
\end{proof}

\begin{rema}\label{rema:s}
  Le th\'eor\`eme est d\'emontr\'e pour la valeur sp\'ecifique de $s$ \'ecrite,
  mais en r\'ealit\'e, on peut r\'egulariser les solutions obtenues par un
  diff\'eomorphisme pour obtenir une param\'etrisation dans $\cM^{s,\ell}_{\delta,
    \setC}$ pour $s\gg \ell$. (Cela correspond \`a l'id\'ee qu'une forme holomorphe
  symplectique d\'etermine des coordonn\'ees holomorphes dans lesquelles
  elle s'exprime avec des coefficients $C^\infty$). Cette r\'egularisation se
  fait en r\'esolvant un probl\`eme de jauge sur le diff\'eomorphisme, ce
  qui est possible en \'etendant l'analyse de la section
  \ref{sec:analyse} du cas des fonctions au cas des champs de vecteurs
  ; pour \'eviter d'allonger inutilement l'article, on a donc pr\'ef\'er\'e se
  limiter \`a cet \'enonc\'e. Pr\'ecisons \`a nouveau qu'\`a la fin, on sait bien
  par la construction twistorielle que la m\'etrique hyperk\"ahl\'erienne
  est $C^\infty$ dans l'int\'erieur de $M_0$.
\end{rema}

\begin{rema}
  On s'attendrait \`a ce que la r\'egularit\'e des bords des domaines construits soit $\FS^{\ell+1+\delta}$, c'est-\`a-dire la r\'egularit\'e de $\xi$. Mais la forme de Liouville du cotangent, $\Theta$, est r\'ecup\'er\'ee en int\'egrant la forme holomorphe symplectique $\Omega$ le long des fibres de la projection $\pi$, ce qui ne permet pas de gain de r\'egularit\'e, et il en r\'esulte qu'a priori la r\'egularit\'e des bords est seulement $\FS^{\ell+\delta}$. Cela illustre le probl\`eme de perte de d\'eriv\'ees que notre m\'ethode a permis de contourner.
\end{rema}

\section{La composante de Hitchin pour $SL(\infty,\setR)$}
\label{sec:la-composante-de}

Dans cette section, on se limite \`a nouveau \`a $\delta=\frac12$ et
$s=\ell+\frac12$. La notation g\'en\'erale est laiss\'ee car les \'enonc\'es sont
valables en r\'ealit\'e pour $s$ plus grand, voir remarque \ref{rema:s}, mais
nous avons limit\'e la d\'emonstration \`a ce cas particulier.

Le th\'eor\`eme \ref{theo:comp-H} montre l'existence d'une param\'etrisation $$\mho:\FS^{\ell+2+\delta}\to\cC^{s,\ell}_\delta\subset \cM^{s,\ell}_{\delta, \setC},$$ telle que
\begin{equation}
 d\mho(g) = \cL_\xi\Omega^0, \quad \xi = 2\big(igw\partial_w-\sharp\db_Hg\big),\label{eq:81}
\end{equation}
o\`u $g$ a \'et\'e \'etendue holomorphiquement disque \`a disque.
En particulier, \'ecrivant $g=g_1+ig_2$, la modification infinit\'esimale du domaine $M_s\subset T^*\Sigma$
est donn\'ee par le d\'eplacement infinit\'esimal de $X\subset T^*\Sigma$ par le vecteur $\re \xi$ le long de $X$ :
\begin{equation}
 \re \xi|_X = \xi_1+\xi_2, \quad
 \begin{cases}
   \xi_1&=-g_1X_1+(X_3g_1)X_2-(X_2g_1)X_3,\\
   \xi_2&=g_2x^{-1}\partial_x+(X_2g_2)X_2+(X_3g_2)X_3 .
 \end{cases}\label{eq:82}
\end{equation}

\begin{defi}
  La composante de Hitchin pour $SL(\infty,\setR)$ est constitu\'ee des couples
  $(\omega^c,\omega_1)\in \cM^{s,\ell}_\delta$ tels que $\omega^c\in \cC^{s,\ell}_\delta$, c'est-\`a-dire des
  domaines de $T^*\Sigma$ qui portent des m\'etriques hyperk\"ahl\'eriennes pli\'ees.
\end{defi}
On remarquera que restreindre $\omega^c$ \`a $\cC^{s,\ell}_\delta$ tue ipso facto
l'ambigu\"\i t\'e de jauge d\^ue \`a l'action des diff\'eomorphismes.

D'autre part, la r\'egularit\'e finie n'est pr\'esente dans cette d\'efinition
que pour des raisons techniques, on ne s'int\'eresse en r\'ealit\'e qu'\`a des
objets lisses. Mais les solutions ne sont pas forc\'ement lisses
jusqu'au bord, d'o\`u la n\'ecessit\'e d'en fixer la r\'egularit\'e.

Rappelons que la d\'ecomposition en s\'eries de Fourier des fonctions CR
holomorphes sur $X$ est $\oplus_{n\geq 0}H^0(\Sigma,K^n)$ (une section de $K^n$
correspond \`a une fonction de fr\'equence $-n$).
\begin{theo}\label{theo:comp-H}
  Pr\`es de la structure standard, la composante de Hitchin pour
  $SL(\infty,\setR)$ est une sous-vari\'et\'e de $\cC^{s,\ell}_\delta$ dont l'espace tangent
  s'identifie \`a
  $$  \FS^{\ell+2+\delta}(X) \cap \oplus_{n\geq 2}H^0(\Sigma,K^n), $$
  c'est-\`a-dire aux fonctions CR holomorphes sur
  $X$ de r\'egularit\'e $\FS^{\ell+2+\delta}$, modulo $\setC \oplus H^0(\Sigma,K)$.

  En particulier, pr\`es de la structure standard, les \'el\'ements de la composante de Hitchin sont d\'etermin\'es par leurs polyn\^omes invariants d\'efinis par (\ref{eq:105}).
\end{theo}
\begin{rema}\label{rema:section-H}
  Au vu de la remarque \ref{rem:pol-inv}, quitte \`a appliquer un
  diff\'eomorphisme, on peut supposer que la param\'etrisation de la
  composante de Hitchin,
  $$ \FS^{\ell+2+\delta}(X) \cap \oplus_{n\geq 2}H^0(\Sigma,K^n) \longrightarrow \cM^{s,\ell}_\delta, $$
  est une section de l'application de Hitchin qui \`a une solution
  associe ses polyn\^omes invariants.
\end{rema}
\begin{proof}
  Il s'agit de montrer que la restriction \`a $\cC^{s,\ell}_\delta$ de
  l'op\'erateur $P$ consid\'er\'e dans les sections \ref{sec:lesp-des-metr}
  et \ref{sec:espaces-fonctionnels}, \`a savoir
  $$ Q : \{ (\omega_1,\omega^c)\in \cQ^{s,\ell}_\delta \text{ tel que }\omega^c\in \cC^{s,\ell}_\delta
  \} \longrightarrow \cH^{s,\ell}_{\delta;\zeta}(\setR^3) $$
  d\'efini par $$Q(\omega_1,\omega^c)=(\omega_1\land \omega^c,\omega_1^2-\frac12 \omega^c\land \overline{\omega^c}),$$
  est une submersion. Ici l'indice $\zeta$ signifie que l'image est restreinte
  au sous-espace
  $$ \int_{M_0} \omega_1\land \omega^c = \zeta_1 \cup (\zeta_2+i\zeta_3). $$
  L'op\'erateur $Q$ est clairement la restriction de l'op\'erateur $P$
  d\'efini dans la section \ref{sec:lesp-des-metr}.

  L'analyse est d\'ej\`a faite dans la section \ref{sec:les-deform-infin},
  o\`u l'on a vu qu'en la m\'etrique standard, on peut toujours r\'esoudre
  le syst\`eme $dQ(d\eta_1,d\eta^c)=(v_1,v^c)$ par une solution du
  type
  \begin{equation}
   (\eta_1 = \eta_{1;0} + d^Cf - \imath_{\xi_2} \omega_1, \eta^c = -\imath_{\xi_2} \omega^c),\label{eq:83}
 \end{equation}
  avec $\xi_2$ un champ de vecteurs du type (\ref{eq:57}), \`a savoir
  \begin{equation}
   \xi_2 = \psi x^{-1}\partial_x + (X_2\psi)X_2 + (X_3\psi)X_3.\label{eq:84}
 \end{equation}
  Le probl\`eme dans (\ref{eq:83}) est que $\imath_{\xi_2}\omega^c$ n'a pas de
  raison d'\^etre un vecteur tangent \`a $\cC^{s,\ell}_\delta$.

  Or dans (\ref{eq:83}) le seul fait important est la valeur de $\psi$
  au bord (dans $\FS^{\ell+2+\delta}$), peu importe son extension \`a l'int\'erieur (en effet, la
  valeur au bord est l\`a pour compenser la singularit\'e de $d^Cf$). Donc
  le r\'esultat reste valable avec le choix suivant d'extension : on
  choisit $\xi_2$ dans (\ref{eq:82}) avec $g_2|_X=\psi$, et on compl\`ete
  avec une fonction $g_1$ telle que $g=g_1+ig_2$ soit \`a fr\'equences
  positives sur $X$, et $\re \int_X g_1=0$. \'Etendant $g$ comme une
  fonction holomorphe disque \`a disque, les champs de vecteurs $\xi_1$ et
  $\xi_2$ sont ainsi \'etendus sur $M_0$, et l'action infinit\'esimale de
  $\xi_1+\xi_2$ est tangente \`a $\cC^{s,\ell}_\delta$.

  La forme $\eta^c=-\imath_{\xi_2}\omega^c$ n'est pas encore tangente \`a
  $\cC^{s,\ell}_\delta$, mais, puisque $\xi_1|_X$ est un champ de vecteurs de
  contact, nous pouvons appliquer l'action infinit\'esimale de $\xi_1$
  sans sortir des espaces fonctionnels, donc utiliser \`a la place de
  (\ref{eq:83}) la solution
  \begin{equation}
    \label{eq:85}
    (\eta_{1;0} + d^Cf - \imath_{\xi_2+\xi_1} \omega_1, -\imath_{\xi_2+\xi_1} \omega^c).
  \end{equation}
  Maintenant $\imath_{\xi_1+\xi_2}\omega^c$ est un vecteur tangent \`a $\cC^{s,\ell}_\delta$, ce qui
  prouve la surjectivit\'e de $dQ$.

  La composante de Hitchin est donc une sous-vari\'et\'e de $\cC^{s,\ell}_\delta$
  dont l'espace tangent est donn\'e par les d\'eformations infinit\'esimales
  du syst\`eme, analys\'ees sections \ref{sec:les-deform-infin} et
  \ref{sec:lespace-des-modules-1}. Le th\'eor\`eme d\'ecoule alors de la
  proposition \ref{prop:Hitchin-infinitesimal} (les d\'eformations
  construites dans cette proposition ne sont pas dans la jauge fournie
  par $\cC^{s,\ell}_\delta$, mais bien entendu les constructions de la section
  \ref{sec:param-des-doma} montrent qu'on peut les y ramener).

  Enfin, l'assertion sur les polyn\^omes invariants d\'ecoule imm\'ediatement du calcul fait dans la remarque \ref{rem:pol-inv}.
\end{proof}

\section{Analyse}
\label{sec:analyse}

Nous d\'emontrons dans cette section les propri\'et\'es de base du laplacien
scalaire pour la g\'eom\'etrie induite par une m\'etrique hyperk\"ahl\'erienne
pli\'ee $g_0$. Notons $\Delta_0=x\Delta$ et rappelons (\ref{eq:32}) :
\begin{equation}
  \label{eq:86}
  \Delta_0 = - \partial_x^2 - X_2^2 - X_3^2 - x^2X_1^2 + x^2 F(\partial_x,xX_1,X_2,X_3),
\end{equation}
o\`u $F$ est un op\'erateur \`a coefficients lisses jusqu'au bord.

Les propri\'et\'es hypoelliptiques \cite{FolSte74} du laplacien $\square =
-X_2^2 - X_3^2$ permettent de d\'ecomposer une fonction $f$ sur $X$ suivant les
valeurs propres $\lambda^2$ de $\square$ :
\begin{equation}
  \label{eq:87}
  f = \sum_\lambda f_\lambda .
\end{equation}
L'espace de Folland-Stein $\FS^s$ est alors d\'efini par la norme $\sum
(1+\lambda^s) \|f_\lambda\|^2$.

\subsection{Le laplacien mod\`ele dans les espaces de Sobolev ordinaires}
\label{sec:le-laplacien-modele}

Pla\c cons-nous dans le cas mod\`ele o\`u $X$ est un quotient compact du
groupe de Heisenberg, avec la structure d\'ecrite dans la section
\ref{sec:la-geometrie-au}. Ainsi, la formule (\ref{eq:86}) devient
exacte ($F=0$). Consid\'erons la vari\'et\'e \`a bord $(M_0,g_0)$ d\'efinie par
$$ M_0 = [0,1] \times X, \quad g_0 = x(dx^2+(\theta^2)^2+(\theta^3)^2)+x^{-1}(\theta^1)^2. $$
Le laplacien $\Delta_0$ ressemble beaucoup \`a un laplacien ordinaire sur une
vari\'et\'e \`a bord, \`a la diff\'erence pr\`es que le laplacien sur le bord est
remplac\'e par le laplacien hypoelliptique $\square$. N\'eanmoins, il se
comporte de mani\`ere analogue au laplacien standard dans les espaces de
Sobolev ordinaires associ\'es, comme nous allons le voir maintenant.

Nous notons
$$ \bv = x^{-1}\vol^{g_0} = dx\land \theta^1\land \theta^2\land \theta^3 , $$
et utiliserons l'espace $L^2(\bv)$, ainsi que les espaces de Sobolev
naturellement associ\'es,
$$ H^k = \{ f, D^j f \in L^2(\bv) \text{ pour tout }j\leq k \}. $$

\begin{lemm}\label{lem:laplacien-Sob-ord}
  L'op\'erateur $\Delta_0:H^{k+2}(M_0) \to H^k(M_0)$ est un isomorphisme pour
  les deux choix suivants de condition au bord :
  \begin{enumerate}
  \item condition de Dirichlet sur les deux bords $x=0$ et $x=1$ ;
  \item condition de Dirichlet sur le bord $x=1$ et Neumann sur le
    bord $x=0$.
  \end{enumerate}
\end{lemm}
\begin{proof}
  L'avantage du mod\`ele plat est que le champ de vecteurs $X_1$
  engendre une action de cercle, qui commute au laplacien horizontal
  $\square$ ; aussi l'op\'erateur $\Delta_0$ se diagonalise-t-il compl\`etement
  en d\'ecomposant en outre par rapport aux coefficients de Fourier de
  l'action de cercle :
  \begin{equation}
    \Delta_0 = - \partial_x^2 + \lambda^2 + n^2 x^2. \label{eq:88}
  \end{equation}
  Par ailleurs, l'identit\'e $X_1=-[X_2,X_3]$ impose $|n|\leq \lambda^2$.

  Analyser le comportement d'un tel op\'erateur est \'el\'ementaire. Pour
  chaque $(\lambda,n)$, l'existence d'une solution unique \`a l'\'equation
  $\Delta_0f=g$ avec les conditions au bord prescrites est imm\'ediate, et il
  faut donc montrer une estimation $H^{s+2}$ sur $f$. L'outil
  essentiel est l'int\'egration par parties
  \begin{align*}
    \int_0^1 |\partial_xf|^2 + (\lambda^2+n^2x^2)|f|^2
    &= \int_0^1 fg \\
    &\leq \frac12 \int_0^1 (\lambda^2+n^2x^2)|f|^2 + \frac{|g|^2}{\lambda^2+n^2x^2}
  \end{align*}
  d'o\`u r\'esulte
  \begin{equation}
    \label{eq:89}
    \int_0^1 |\partial_xf|^2 + \frac12 (\lambda^2+n^2x^2)|f|^2
    \leq \frac12 \int_0^1 \frac{|g|^2}{\lambda^2+n^2x^2}.
  \end{equation}
  En particulier,
  \begin{equation}
    \label{eq:90}
    \int_0^1 \lambda^{2s} |\partial_xf|^2 + \lambda^{2s+4}|f|^2 \leq \lambda^{2s} \int_0^1 |g|^2
  \end{equation}
  donne d\'ej\`a le contr\^ole voulu des d\'eriv\'ees de $f$ et $\partial_xf$ suivant
  $X_2$ et $X_3$.

  On a l'\'equation
  \begin{equation}
    \label{eq:91}
    \Delta_0(xf) = xg - 2 \partial_xf .
  \end{equation}
  Appliquant (\ref{eq:89}) \`a $xf$, on obtient
  \begin{equation*}
    \int_0^1 |\partial_x(xf)|^2 + n^2x^4|f|^2 \leq \int_0^1 \frac{|xg|^2+4|\partial_xf|^2}{\lambda^2+n^2x^2}
    \leq \int_0^1 (n^{-2} + 2 \lambda^{-4}) |g|^2,
  \end{equation*}
  o\`u, dans la seconde in\'egalit\'e, on a r\'eutilis\'e (\ref{eq:89}) pour
  $f$.  Puisque $n^2\leq \lambda^4$, on obtient ainsi
$$ \int_0^1 \frac12 n^2x^2|\partial_xf|^2 + n^4x^4|f|^2 \leq \int 3|g|^2 + n^2|f|^2 \leq
4\int |g|^2; $$ l'\'equation $\Delta_0f=g$ permet alors d'estimer aussi
$\partial_x^2f$, et on a obtenu ainsi une estimation sur toutes les d\'eriv\'ees
secondes de $f$, donc
\begin{equation}
  \label{eq:92}
  \| f \|_{H^2} \leq c \| g \|_{L^2}.
\end{equation}

L'estimation plus g\'en\'erale $\| f\|_{H^{k+2}}\leq c_k \| g\|_{H^k}$ est alors \'etablie par r\'ecurrence sur $s$ : d'une part, les d\'eriv\'ees suivant $X_2$ et $X_3$ ont d\'ej\`a \'et\'e born\'ees, les d\'eriv\'ees suivant $X_1$ sont successivement born\'ees en utilisant comme ci-dessus
$$ \Delta_0(x^jf) = x^j g - 2jx^{j-1}\partial_xf - j(j-1)x^{j-2}f, $$
et les d\'eriv\'ees $\partial_x^jf$ en d\'erivant l'\'equation. Les d\'etails sont laiss\'es au lecteur, car le cas $k>0$ n'est de toute fa\c con pas utilis\'e dans la suite.
\end{proof}

\subsection{Le laplacien dans les espaces de Sobolev ordinaires}
\label{sec:le-laplacien-dans}

Les r\'esultats de la section pr\'ec\'edente se g\'en\'eralisent \`a une vari\'et\'e hyperk\"ahl\'erienne pli\'ee g\'en\'erale, $(M_0,g_0)$, \`a bord $X$. Les espaces de Sobolev ordinaires $H^k$ restent d\'efinis en prenant les normes $L^2$ par rapport \`a une forme volume $\bv=x^{-1}\vol^{g_0}$ pr\`es du bord.
\begin{prop}\label{prop:laplacien-Sob-ord}
  Le laplacien $\Delta=x^{-1}\Delta_0$ est un isomorphisme :
  \begin{enumerate}
  \item $H^{k+2}(M_0)\to x^{-1}H^k(M_0)$, avec condition de Dirichlet \`a la source ;
  \item $H^{k+2}_0(M_0)\to(x^{-1}H^k(M_0))_0$, avec condition de Neumann \`a la source, et o\`u l'indice $0$ signifie qu'on se limite aux fonctions d'int\'egrale nulle : $\int_{M_0}f\vol^{g_0}=0$.
  \end{enumerate}
\end{prop}
\begin{proof}
  Le probl\`eme est de passer du lemme \ref{lem:laplacien-Sob-ord} dans le cas mod\`ele \`a l'\'enonc\'e du th\'eor\`eme. C'est une m\'ethode classique sur laquelle nous ne donnerons pas beaucoup de d\'etails. L'\'enonc\'e r\'esulte imm\'ediatement de la construction d'un parametrix pour $\Delta$, dans chacun des deux cas. La m\'ethode consiste \`a voir qu'en tout point du bord $X$, la g\'eom\'etrie de $g$ est bien approch\'ee par celle du mod\`ele, gr\^ace \`a (\ref{eq:16}). On peut donc fabriquer un inverse approximatif de $\Delta$ en recouvrant $X$ par un nombre assez grand de petits ouverts o\`u la structure hyperk\"ahl\'erienne pli\'ee est tr\`es proche de celle du mod\`ele. L'inverse construit pour le mod\`ele peut alors \^etre greff\'e sur $M_0$ pour donner un inverse approximatif sur un voisinage de $X$ dans $M_0$ ; compl\'et\'e par un parametrix sur l'int\'erieur, il fournit le parametrix attendu. C'est exactement la m\'ethode utilis\'ee dans \cite[chapitre I]{Biq00} pour l'analyse des m\'etriques asymptotiquement hyperboliques complexes, qui offrent comme on l'a vu une g\'eom\'etrie tr\`es similaire.
\end{proof}

\begin{rema}\label{rem:N}
  De la m\^eme mani\`ere, on peut se placer sur un petit voisinage $N_\epsilon=(0,\epsilon)\times X$ de $X$ dans $M_0$, avec, comme dans la section \ref{sec:le-laplacien-modele} une condition de Dirichlet sur le bord $x=\epsilon$, et on obtient alors, tant avec la condition de Dirichlet qu'avec la condition de Neumann en $x=0$, un isomorphisme
  \begin{equation}
\Delta:H^{k+2}(N_\epsilon)\longrightarrow x^{-1}H^k(N_\epsilon).\label{eq:93}
\end{equation}
\end{rema}

\subsection{Le laplacien dans les espaces de Sobolev \`a poids}
\label{sec:le-laplacien-modele-1}

Les espaces de Sobolev ordinaires ne sont malheureusement pas suffisants pour contr\^oler la non-lin\'earit\'e de l'\'equation $P(\alpha)=0$ : \`a cause du volume $\vol^{g_0}\sim x^{-2} dx\land e^1\land e^2\land e^3$, le terme quadratique $(d\alpha_a\land d\alpha_b)$ donne lieu \`a des termes $x^{-2}D\alpha_{a,i}D\alpha_{b,j}$ qui sont trop singuliers. C'est la raison de l'utilisation d'espaces \`a poids.

Il existe des liens entre les espaces de Sobolev ordinaires et \`a poids demi-entiers : il est clair que
\begin{equation}
  \label{eq:94}
  L^2_{-\frac 12} = L^2(\bv) ,
\end{equation}
et en outre, on a l'\'egalit\'e des deux espaces,
\begin{equation}
  \label{eq:95}
  \cH^2_{1,\frac 12} = H^2,
\end{equation}
car ils sont tous deux caract\'eris\'es par la condition $D^2f\in L^2(\bv)$. La proposition \ref{prop:laplacien-Sob-ord} dit donc d\'ej\`a qu'on obtient des isomorphismes
\begin{align}
  x\cH^2_{\frac 12} & \longrightarrow x^{-1}L^2_{-\frac12}=H^0_{-\frac32} ,\label{eq:96}
\intertext{qui inclut implicitement la condition de Dirichlet sur $X$, et}
  \cH^2_{1,\frac 12;0} & \longrightarrow H^0_{-\frac32;0} \label{eq:97}
\end{align}
avec condition de Neumann \`a la source (et l'indice $0$ signifie qu'on se limite aux fonctions d'int\'egrale nulle). Notons $\cD$ et $\cN$ les inverses respectifs de $\Delta$ dans les deux cas.

\begin{rema}\label{rem:delta}
  On peut modifier la preuve du lemme \ref{lem:laplacien-Sob-ord} pour obtenir des estimations pour un poids $\delta\in (0,1)$ au lieu de $\frac12$ dans (\ref{eq:96}) et (\ref{eq:97}), mais le poids $\delta=\frac12$ suffit pour les r\'esultats de cet article.
\end{rema}

Commen\c cons par voir que ces isomorphismes s'\'etendent quand les d\'eriv\'ees sont contr\^ol\'ees avec des poids :
\begin{lemm}\label{lem:iso-poids12}
  Pour tout $s\geq 0$ on a sur $M_0$ des isomorphismes
  \begin{align}
    x\cH^{s+2}_{\frac12} & \longrightarrow H^s_{-\frac32}, \label{eq:98}\\
    \cH^{s+2}_{1,\frac 12;0} & \longrightarrow H^s_{-\frac32;0}.\label{eq:99}
  \end{align}
\end{lemm}
\begin{proof}
  En utilisant le lemme d'extension \ref{lem:extension}, on est ramen\'e \`a montrer que si $f\in L^2_{\frac32}$ et $\Delta f\in H^s_{-\frac32}$, alors $f\in H^{s+2}_{\frac32}$, ce qui est un \'enonc\'e de r\'egularit\'e elliptique dans les espaces de Sobolev \`a poids. Cette r\'egularit\'e est une cons\'equence de l'homog\'en\'e\"\i t\'e (\ref{eq:16}) : soit $x_0>0$, alors $g'=x_0^{-3}h_{x_0}^*g$ est proche de la m\'etrique mod\`ele $G$ construite \`a partir du groupe de Heisenberg, et en particulier une boule de rayon $\rho x_0^{\frac32}$ est ainsi envoy\'ee sur une boule de rayon $\rho$ pour $g'\approx G$. On en d\'eduit qu'on dispose d'estimations elliptiques dans la boule $B_\rho(g')$ de rayon $\rho$ pour $g'$, avec constante ind\'ependante du point choisi :
$$ \|f\|_{H^{s+2}(B_\rho(g'))}^2 \leq c \left( \|f\|_{L^2(B_\rho(g'))}^2 + \|\Delta^{g'}f\|_{H^s(B_\rho(g'))}^2 \right). $$
Appliquons le changement d'\'echelle qui ram\`ene \`a $g$ : on obtient
\begin{multline*}
  \sum_0^{s+2} \|x_0^{\frac{3j}2}D^jf\|_{L^2(B_{\rho x_0^{3/2}}(g))}^2 \\ \leq c
  \left( \|f\|_{L^2(B_{\rho x_0^{3/2}}(g))}^2 + \sum_0^s
    \|x_0^{3+\frac{3j}2}\Delta^gf\|_{L^2(B_{\rho x_0^{3/2}}(g))}^2 \right).
\end{multline*}
Comme $x$ et $x_0$ sont comparables dans la boule $B_{\rho x_0^{3/2}}(g)$ si $\rho$ a \'et\'e fix\'e assez petit, c'est l'estimation \`a poids voulue, apr\`es multiplication par $x_0^{-5}$ pour avoir les bons poids. Il suffit ensuite de sommer sur un recouvrement localement fini par des boules de rayon $\rho x^{3/2}$.
\end{proof}

Analysons les inverses $\cD$ et $\cN$ de (\ref{eq:98}) et (\ref{eq:99}) sur l'espace $\cH^s_{\frac12}$ :
\begin{prop}\label{prop:laplacien-mod-poids}
  Supposons $s\geq \ell+\frac12$ et $g\in \cH^{s,\ell}_{\frac 12}$ (resp. $g\in \cH^{s,\ell}_{\frac 12;0}$). Alors $\cD g$ (resp. $\cN g$) est dans l'espace $\cH^{s+2,\ell}_{3,\frac12}$, et $\cD:\cH^{s,\ell}_{\frac 12}\to x\cH^{s+2,\ell}_{2,\frac12}$ (resp. $\cN:\cH^{s,\ell}_{\frac 12;0}\to\cH^{s+2,\ell}_{3,\frac12;0}$) est un op\'erateur continu.
\end{prop}

\begin{proof}
  On utilise la commutation de $\square$ et $X_1$ avec $\Delta$. Dans le cas mod\`ele, on a exactement $[X_1,\Delta]=0$ et $[\square,\Delta]=0$. En g\'en\'eral, observons que $\theta^1([X_1,X_2]) = -d\theta^1(X_1,X_2) = 0$, donc $[X_1,X_2]$ est horizontal, et de m\^eme $[X_1,X_3]$. Il en r\'esulte $[X_1,\square] = \partial^2$, o\`u par $\partial^2$ nous entendons n'importe quel op\'erateur sur $X$, d'ordre 2 en les d\'erivations horizontales, et finalement
  \begin{align}
    \label{eq:100}
    [X_1,\Delta] &= x^{-1} O_2 , \\
    [\square,\Delta] &= O_3,\label{eq:101}
  \end{align}
  o\`u $O_j$ est un op\'erateur d'ordre $j$ en les d\'erivations $e_i$.

  Commen\c cons par le probl\`eme de Dirichlet. Supposons $g\in
  x\cH^s_{\frac12}$ et posons $f=\cD g\in x\cH^{s+2}_{\frac12}$. \`A
  partir de $\square g\in H^{s-2}_{-\frac12}$ et $[\square,\Delta]f\in
  H^{s-1}_{-\frac32}$ par (\ref{eq:101}), le lemme
  \ref{lem:iso-poids12}, appliqu\'e \`a un voisinage $N_\epsilon$ de $X$ dans $M_0$ comme dans la remarque \ref{rem:N}, donne (avec estimation)
  \begin{equation}
  \square f\in x\cH^s_{\frac12}.\label{eq:102}
  \end{equation}
  De m\^eme, $X_1g\in H^{s-1}_{-\frac12}$ et $[X_1,\Delta]f\in H^s_{-\frac32}$,
  donc $X_1f\in x\cH^{s+1}_{\frac12}$, d'o\`u on d\'eduit
  \begin{equation}
  x^2X_1^2f\in x\cH^s_{\frac 12}.\label{eq:103}
  \end{equation}
  De (\ref{eq:102}), (\ref{eq:103}) et de l'\'equation $\Delta f=g$ on d\'eduit
  enfin $\partial_x^2f\in x\cH^s_{\frac 12}$, donc on a pour toutes les
  d\'eriv\'ees secondes (avec estimation en fonction de $\|g\|_{x\cH^s_{\frac12}}$)
  \begin{equation}
    \label{eq:104}
    D^2f \in x\cH^s_{\frac 12}.
  \end{equation}
  Il s'ensuit que $f\in \cH^{s+2}_{2,\frac12}$, et m\^eme, puisque $f$
  satisfait la condition de Dirichlet, $f\in x\cH^{s+2}_{1,\frac12}$.

  Le probl\`eme de Neumann est similaire : si $g\in x\cH^s_{\frac12}$ et
  $f=\cN g$, alors le m\^eme raisonnement fournit un contr\^ole $D^2f\in
  \cH^{s-2}_{1,\frac12}$, qui implique $f\in \cH^{s+2}_{2,\frac12}$.

  Finalement, le cas o\`u $\ell\neq 0$ s'en d\'eduit par r\'ecurrence sur $\ell$, en
  appliquant les estimations pr\'ec\'edentes \`a $\square^{\frac \ell2}f$, dont
  les images par $\Delta$ sont contr\^ol\'ees gr\^ace \`a la commutation \'evidente
  $[\square^{\frac12},\Delta]=x^{-1}O_2$ (en fait, on peut m\^eme montrer que ce
  commutateur est $O_2$).
\end{proof}

\providecommand{\bysame}{\leavevmode ---\ }
\providecommand{\og}{``}
\providecommand{\fg}{''}
\providecommand{\smfandname}{\&}
\providecommand{\smfedsname}{\'eds.}
\providecommand{\smfedname}{\'ed.}
\providecommand{\smfmastersthesisname}{M\'emoire}
\providecommand{\smfphdthesisname}{Th\`ese}

\end{document}